\documentclass[12pt]{amsart}

\usepackage[hmargin=0.8in,height=8.6in]{geometry}
\usepackage{amssymb,amsthm, times,graphicx}
\usepackage{delarray,verbatim}
\usepackage{caption}
\usepackage{subcaption}
\usepackage{bm}

\usepackage{color}
\definecolor{webgreen}{rgb}{0,.5,0}
\definecolor{webbrown}{rgb}{.8,0,0}
\definecolor{emphcolor}{rgb}{0.95,0.95,0.95}
\usepackage{hyperref}
\hypersetup{%
	colorlinks=true,
	linkcolor=webbrown,
	filecolor=webbrown,
	citecolor=webgreen,
	breaklinks=true}
\ifpdf \hypersetup{pdftex,
	pdfstartview=FitH, %%Fit, FitB, FitH
	bookmarksopen=true,
	bookmarksnumbered=true
} \else \hypersetup{dvips} \fi

\newcommand {\ud}{{\rm d}}

\DeclareMathOperator{\hol}{C}

\DeclareMathOperator{\expo}{e}

\numberwithin{equation}{section}

\numberwithin{equation}{section}
\theoremstyle{plain}
\newtheorem{theorem}{Theorem}[section]

\newtheorem{proposition}[theorem]{Proposition}
\newtheorem{lemma}[theorem]{Lemma}

\newtheorem{assump}[theorem]{Assumption}
\newtheorem{remark}[theorem]{Remark}

\newcommand {\R}{\mathbb{R}}

\newcommand {\p}{\mathbb{P}}

\newcommand {\E}{\mathbb{E}}

\newcommand{\diff}{{\rm d}}

\newcommand{\lev}{L\'{e}vy }

\newcommand{\e}{\mathbb{E}}

\title{Optimal bail-out dividends problem with transaction cost and capital injection constraint}
\author[M. Junca]{Mauricio Junca$^{(1)}$}
\thanks{$(1)$ Department of Mathematics, Universidad de los Andes, Bogot\'a, Colombia. Email: mj.junca20@uniandes.edu.co}
\author[H. Moreno-Franco]{Harold Moreno-Franco$^{(2)}$}
\thanks{$(2)$ Department of Mathematics and Statistics, Universidad del Norte, Barranquilla, Colombia. Email:hamoreno@uninorte.edu.co}
\author[J.L. P\'erez]{Jos\'e Luis P\'erez$^{(3)}$}
\thanks{$(3)$ Department of Probability and Statistics, Centro de Investigaci\'on en Matem\'aticas A.C.,Guanajuato, Mexico. Email: jluis.garmendia@cimat.mx.}
\begin{document}
	\maketitle	
	
	\begin{abstract}
		We consider the bail-out optimal dividend problem under fixed transaction costs for a L\'evy risk model with a constraint on the expected net present value of injected capital. In order to solve this problem, we first consider the bail-out optimal dividend problem under transaction costs and capital injection and show the optimality of reflected $(c_1,c_2)$-policies. We then find the optimal Lagrange multiplier, by showing that in the dual Laagrangian problem, the complementary slackness conditions are verified. Finally, we verify our results with some numerical examples.
		
	%	Next, we introduce the dual Lagrangian problem and show that the complementary slackness conditions are satisfied, characterizing the optimal Lagrange multiplier. Finally, we illustrate our findings with a series of numerical examples.\\
		\noindent \small{\noindent  AMS 2010 Subject Classifications: 60G51, 93E20, 91B30 
		}
		
	\end{abstract}

\section{Introduction}
In the bail-out model of de Finetti's optimal dividend problem, one wants to maximize the total expected dividends minus the costs of capital injection under the constraint that the surplus must be kept non-negative uniformly in time. Typically, a spectrally negative L\'evy process (a L\'evy process with only downward jumps) is used to model the underlying surplus process of an insurance company that increases because of premiums and decreases by insurance payments. Avram et al. \cite{AvPaPi07} showed that it is optimal to reflect
from below at zero and also from above at a suitably chosen threshold. 

We focus on the extension for which a transaction cost is incurred each time a dividend payment is made. Because of this fixed cost, it is no longer feasible to pay out dividends at a certain rate and therefore only lump sum dividend payments are possible. In this case, a strategy is assumed to be in the form of impulse control;
whenever dividends are accrued, a constant transaction cost $\delta > 0$ is incurred. As opposed to the barrier strategy that is typically
optimal for the no-transaction cost case, we shall pursue the
optimality of the reflected $(c_1, c_2)$-policy that brings the surplus process down to $c_1$ whenever it reaches or exceeds $c_2$ for some $0 \leq c_1 < c_2 < \infty$, and pushes the surplus to $0$ whenever it goes below $0$. Regarding the version without bail-outs, the de Finetti's optimal dividend problem under fixed transaction costs was solved for the spectrally negative case by Loeffen \cite{LoeffenTrans} and for the dual model by Bayraktar et al. \cite{BayraktarImpdual}. 
\par In this work we are interested in studying the case in which the longevity aspect of the firm is considered, by adding a constraint on expected net present value of injected capital. Similar studies have recently been done in this direction by Hern\'andez et al. \cite{HJM17} (see also \cite{HJPY18} for the case with absolutely continuous strategies). Following \cite{schmidli2002}, the performance and longevity of the firm remained as two separate problems. Although there exist a series of works which tried to integrate both features \cite{Jostein03,ThonAlbr,Grandits}, it was not until very recently that Hern\'andez and Junca \cite{HJ15} presented a solution which succesfully took into account both elements. In their work they considered a Cram\'er-Lundberg process with i.i.d. exponentially distributed jumps as the model for the reserves and added a cosntraint to the expected time of ruin of the company.

%\par Following the recent work of Hern\'andez et al. \cite{HJM17} (see also \cite{HJPY18} for the case with absolutely continuous strategies), we study the case in which the longevity feature is added to the problem  by considering a constraint on expected net present value of injected capital. The longevity aspect of the firm remained as a separate problem; see \cite{schmidli2002} for a survey on this matter. Despite efforts to integrate both features \cite{Jostein03,ThonAlbr,Grandits}, it was not until very recently  a successful solution to a model that actually accounts for the trade-off between performance and longevity was presented. Hern\'andez and Junca \cite{HJ15} considered de Finetti's problem in the setting of Cram\'er-Lundberg reserves with i.i.d. exponentially distributed jumps adding a constraint to the expected time of ruin of the firm. 

In this paper we solve the following two problems.
\begin{enumerate}
	\item First we find the solution to the bail-out optimal dividend problem under transaction costs. We solve this problem for the spectrally negative. We show that a reflected $(c_1,c_2)$ policy is optimal (see Lemma \ref{L.V.1}). We use scale functions to characterize the optimal thresholds as well as the value function. We show the optimality of the proposed policy by means of a verification lemma.
	\item We then solve the constrained dividend maximization problem with capital injection over the set of strategies such that the expected present value of injected capital must be bounded by a given constant. This is an offshoot of \cite{HJM17} for the bail-out case. Using the previous results, in Theorems \ref{strdualnocost} and \ref{main.1} we present the solution when the surplus of the company is modeled by a spectrally negative L\'evy process.
	% We then use these results to solve the constrained dividend maximization problem over the set of strategies such that the expected present value of injected capital must be bounded by a given constant.  This is an offshoot of \cite{HJM17} for the bail-out case. Theorems \ref{strdualnocost} and \ref{main.1} show the results when the reserves are modeled by a spectrally negative L\'evy process. 
\end{enumerate}

The organization of the paper is given as follows. In Section \ref{F1}, we introduce the problem. In Section \ref{section_scale_functions}, we provide a review of scale functions and some fluctuation identities of spectrally negative \lev processes  as well as their reflected versions. In Section  \ref{neg.1}, we solve the bail-out optimal dividend problem under fixed  transaction costs for the case of a spectrally negative L\'evy process.  In Section \ref{const.1}, we present the solution for the constrained dividends problem. Finally, in Section \ref{numerical_section}, we illustrate our main results by giving some numerical examples.
\section{Formulation of the problem}\label{F1}

Let $X=\{X_t: t\geq 0\}$  be a \lev process defined on a  probability space $(\Omega, \mathcal{F}, \p)$.  For $x\in \R$, we denote by $\p_x$ the law of $X$ {when it starts at $x$,} and write for convenience  $\p$ in place of $\p_0$. The expectation operators associated with these probabilities  are denoted by $\e_x$ and $\e$, respectively, and let us define  $\mathbb{F}:=\{\mathcal{F}_t:t \geq 0\}$ as the completed and right-continuous filtration generated by $X$. 

In this paper, we  assume throughout that $X$ is \textit{spectrally negative},   meaning here that it has no positive jumps and is not the negative of a subordinator. We define the Laplace exponent 
\[
\e\big[\expo ^{\theta X_t}\big]=:\expo ^{\psi(\theta)t}, \qquad t, \theta\ge 0,
\]
given by the \emph{L\'evy-Khintchine formula}
\begin{equation*}
\psi(\theta):=\gamma\theta+\dfrac{\sigma^2}{2}\theta^2-\int_{(0,\infty)}\big(1-\expo ^{-\theta z}-\theta z\mathbf{1}_{\{0<z\leq1\}}\big)\Pi(\ud z), \quad \theta \geq 0,
\end{equation*}
where $\gamma\in \R$, $\sigma\ge 0$, and $\Pi$ is a measure concentrated on $(0,\infty)$ called the L\'evy measure of $X$ that satisfies
\[
\int_{(0,\infty)}(1\land z^2)\Pi(\ud z)<\infty.
\]
It is well-known that $X$ has paths of bounded variation if and only if $\sigma=0$ and $\displaystyle\int_{(0,1]} z\Pi(\ud z)<\infty$. In this case $X$ can be written as
\begin{equation}
X_t=ct-S_t, \,\,\qquad t\geq 0,\notag
\end{equation}
where $c:=\gamma+\displaystyle\int_{(0,1]} z\Pi(\ud z)$ and $(S_t; t\geq0)$ is a drift-less subordinator. Note that  necessarily $c>0$, since we have ruled out the case that $X$ has monotone paths; its Laplace exponent is given by
\begin{equation*}
	\psi(\theta) = c \theta-\int_{(0,\infty)}\big(1-\expo ^{-\theta z}\big)\Pi(\ud z), \quad \theta \geq 0.
\end{equation*}
\subsection{Admissible strategies.}
Let $D=\{D_t:t\geq 0\}$ be a dividend strategy, meaning that it is a left-continuous, non-negative, and  non-decreasing process adapted to the filtration $\mathbb{F}$.  For each $t\geq0$, the quantity $D_t$ represents the cumulative dividends paid out up to time $t$ by the insurance company whose risk process is modeled by $X$. Consider the situation where the insurance company is not allowed to go bankrupt and the beneficiaries of the dividends are required to inject capital into the insurance company to ensure its risk process stays nonnegative. Thus, let $R=\{R_t:t\geq 0\}$  be a capital injection strategy, which is  a right-continuous, non-negative and non-decreasing process adapted to the filtration $\mathbb{F}$ describing the cumulative amount of injected capital. We assume the both processes start at 0. Given a pair $\pi=\{D,R\}$ the controlled L\'evy process is thus,
\begin{equation}\label{con.1}
X^{\pi}_t=X_t-D_t+R_t,\qquad\text{$t\geq0$.}
\end{equation}
The set of admissible policies $\Theta$ consists of those policies $\pi$  for which $X^{\pi}$ is non-negative and
\begin{equation}
\int_0^{\infty}\expo^{-qt}\ud R_t<\infty,\qquad\mathbb{P}_x\text{-a.s.},
\end{equation}
that is, the present value, with discounted rate $q>0$, of the injected capital is finite a.s..
\subsection{Constrained de Finetti’s problem with transaction costs and capital injection.}
Given an initial capital $x\geq0$ and a policy $\pi=\{D^{\pi},R^{\pi}\}\in\Theta$ we define the expected net present value (NPV) of dividends minus the costs of capital injection under the
 strategy $\pi$, 
\begin{equation}\label{v.f.strategy}
v^{\pi}_{\Lambda}(x):=\mathbb{E}_x\left[\int_0^{\infty}\expo^{-qt}\diff\left(D_t^{\pi}-\delta\sum_{0\leq s<t}\mathbf{1}_{\{\Delta D_s^{\pi}>0\}}\right)-\Lambda\int_0^{\infty}\expo^{-qt}\diff R_{t}^{\pi}\right],
\end{equation}
where $q>0$ is the discount rate, $\delta\geq0$ is the transaction cost and $\Lambda>0$ is the unit cost per capital injected. When $\delta>0$ we need the dividend strategy $D^{\pi}$ to be a pure jump process. Hence the value function we aim to find is

\begin{equation}\label{v.f}
V_{\Lambda}(x):=\sup_{\pi\in\Theta}v^{\pi}_{\Lambda}(x).
\end{equation}

Since we want to avoid this function to be infinity, we will assume that $\psi^{\prime}(0+)=\mathbb{E}[X_1]>-\infty$. We will also assume that $\Lambda\geq1$, otherwise large amounts of dividends can be paid and bail-out the company by injecting capital at a cheaper cost, so the value function goes to infinity.

\section{Preliminaries}\label{section_scale_functions}

In this section we review the scale function of spectrally negative \lev processes (see \cite{kyprianou2014,KKRivero2013}). We also review known results regarding optimal dividend strategies with capital injection for one-sided spectrally \lev processes when the transaction cost is equal to $0$ (i.e. $\delta=0$).

For each $q \geq 0$, there exists a function $W^{(q)}$, called $q-$scale function which is a mapping from $\R$ to $[0, \infty)$ that takes the value zero on the negative half-line, while on the positive half-line it is a strictly increasing function defined by its Laplace transform:
\begin{equation} \label{scale_function_laplace}
		\int_0^\infty  \mathrm{e}^{-\theta x} W^{(q)}(x) \ud x = \dfrac 1 {\psi(\theta)-q}, \quad \theta > \Phi(q),
\end{equation}
where
\begin{equation}
		\Phi(q) := \sup \{ \lambda \geq 0: \psi(\lambda) = q\}.
\end{equation}

%The function $\psi$ is strictly convex with the property that $\lim\limits_{\theta\rightarrow\infty}\psi(\theta)=+\infty$. Moreover, $\psi$ is strictly increasing on $[\Phi(0),\infty)$.

We also define, for $x \in \R$,
\begin{align*}
	\overline{W}^{(q)}(x) &:=  \int_0^x W^{(q)}(y) \ud y, \\
	Z^{(q)}(x) &:= 1 + q \overline{W}^{(q)}(x),  \\
	\overline{Z}^{(q)}(x) &:= \int_0^x Z^{(q)} (z) \ud z = x + q \int_0^x \int_0^z W^{(q)} (w) \ud w \ud z.
\end{align*}
Since $W^{(q)}(x) = 0$ for $-\infty < x < 0$, we have
\begin{align}
	\overline{W}^{(q)}(x) = 0, \quad Z^{(q)}(x) = 1  \quad \textrm{and} \quad \overline{Z}^{(q)}(x) = x, \quad x \leq 0.  \label{z_below_zero}
\end{align}
\begin{remark}\label{remark_smoothness_zero}
	\begin{enumerate}
		\item By (8.26) of \cite{kyprianou2014}, the left- and right-hand derivatives of $W^{(q)}$ always exists on $\R \backslash \{0\}$.  In addition, as in, e.g., \cite[Theorem 3]{Chan2011}, if $X$ is of unbounded variation or the \lev measure is atomless, we have ${W^{(q)}} \in C^1(\R \backslash \{0\})$.
		\item As in Lemmas 3.1 and 3.2 of \cite{KKRivero2013},
		\begin{align*} 
			\begin{split}
				{W^{(q)}} (0) &= \left\{ \begin{array}{ll} 0, & \textrm{if $X$ is of unbounded
						variation,} \\ \dfrac 1 {c}, & \textrm{if $X$ is of bounded variation,}
				\end{array} \right. \\
				{W^{(q)\prime}} (0+) &
				=
				\left\{ \begin{array}{ll}  \dfrac 2 {\sigma^2}, & \textrm{if }\sigma > 0, \\
					\infty, & \textrm{if }\sigma = 0 \; \textrm{and} \; \Pi(0,\infty) = \infty, \\
					\dfrac {q + \Pi(0,\infty)} {c^2}, &  \textrm{if }\sigma = 0 \; \textrm{and} \; \Pi(0,\infty) < \infty.
				\end{array} \right.
			\end{split}
		\end{align*}
		\item 
		As in Lemma 3.3 of \cite{KKRivero2013}, 
		$W_{\Phi(q)}(x) := e^{-\Phi(q) x}{W^{(q)}} (x) \nearrow \psi'(\Phi(q))^{-1}$, as $x \uparrow \infty$.
	\end{enumerate}
\end{remark}
Due to Remark \ref{remark_smoothness_zero} we will make the following assumption throughout the paper.
\begin{assump}\label{assump_C1}
	We will assume that either $X$ has unbounded variation or $\Pi$ is absolutely continuous with respect to the Lebesgue measure. Under this assumption it holds that $W^{(q)}$ is $\hol^1$ in $(0,\infty)$. 
\end{assump}
For later use, we give the following properties related to the functions $Z^{(q)}$ and $\overline{W}^{(q)}$.
\begin{remark}\label{logconv}
\begin{itemize}
	\item[(i)] By Proposition 5.5 in \cite{HJM17} we have that for $q\geq0$, the function  $Z^{(q)}$ is strictly log-convex on $(0,\infty)$.
\item[(ii)] From Lemma 1 \cite{AvPaPi07}, the function $\overline{W}^{(q)}$ is log-concave in $(0,\infty)$.
\end{itemize}
\end{remark}
Let us define the first down- and up-crossing times, respectively, by
\begin{align}
\label{first_passage_time}
\tau_a^- := \inf \left\{ t > 0: X_t < a \right\} \quad \textrm{and} \quad \tau_a^+ := \inf \left\{ t > 0: X_t >  a \right\}, \quad a \in \R;
\end{align}
here and throughout, let $\inf \varnothing = \infty$. Note that $\tau_a^-$ is the first entrance time of $X$ into $(-\infty,a)$ and $\tau_a^+$ into $(a,\infty)$. Then by Theorem 8.1 in \cite{kyprianou2014}, for any $a > b$ and $x \leq a$,
\begin{align}
\begin{split}
\E_x \left[\expo^{-q \tau_a^+} \mathbf{1}_{\left\{ \tau_a^+ < \tau_b^- \right\}}\right] &= \dfrac {W^{(q)}(x-b)}  {W^{(q)}(a-b)}, \\
\E_x \left[\expo^{-q \tau_b^-} \mathbf{1}_{\left\{ \tau_a^+ > \tau_b^- \right\}}\right] &= Z^{(q)}(x-b) -  Z^{(q)}(a-b) \dfrac {W^{(q)}(x-b)}  {W^{(q)}(a-b)}.
% \E_x \left[ \expo^{-q \tau_0^-} \right] &= Z^{(q)}(x) -  \dfrac q {\Phi(q)} W^{(q)}(x).
\end{split}
\label{laplace_in_terms_of_z}
\end{align}

\subsection{Reflected \lev processes}

Let $S=\{S_{t}:t\geq0\}$ and $R^{0}=\{R^{0}_t:t\geq0\}$ be defined respectively as 
\begin{equation}\label{reflected_levy}
S_t:=\sup\limits_{0\leq s\leq t}(X_s\vee0) \quad \textrm{and} \quad R^{0}_t:=\sup\limits_{0\leq s\leq t}(-X_s\vee0).
\end{equation}
We denote by $\hat{Y}:=S-X$ and $Y:=X+R^{0}$,  which are a strong Markov processes. Observe that the process $R^{0}$ pushes $X$ upward whenever it attempts to down-cross the level $0$; as a result the process $Y$ only takes values on $[0, \infty)$.
The reader is referred to \cite{B,kyprianou2014} for a complete introduction to the theory of L\'evy processes and their reflected processes.  

Given $a>0$, let $\hat{\tau}_a$ be the first entrance time of $\hat{Y}$ into $(a,\infty)$, then by Proposition 2 in \cite{Pistorius2004},
\begin{equation}\label{funH}
\e_{-x}\left[\expo^{-q\hat{\tau}_a}\right]=Z^{(q)}(a-x)-qW^{(q)}(a-x)\dfrac{W^{(q)}(a)}{W^{(q)\prime}(a)},\qquad\text{$x\in[0,a]$.}
\end{equation}
We define for $a>0$,
\begin{equation}\label{defH}
H(a):=\e_{0}\left[\expo^{-q\hat{\tau}_a}\right]=Z^{(q)}(a)-q\dfrac{[W^{(q)}(a)]^{2}}{W^{(q)\prime}(a)}.
\end{equation}
\begin{remark}\label{remH}
Note that by definition the function $H$ is strictly positive,  strictly decreasing and satisfies
\begin{equation*}
\lim_{a\rightarrow\infty}H(a)=0,\quad\quad\lim_{a\rightarrow0}H(a)=1-\dfrac{q[W^{(q)}(0)]^{2}}{W^{(q)\prime}(0+)}.
\end{equation*}
Therefore, the function $H$ has an inverse from $(0,1-q/(q+\Pi(0,\infty))]$ onto $[0,\infty)$ when $\sigma=0$ and $\Pi(0,\infty)<\infty$, and from $(0,1]$ onto $[0,\infty)$ otherwise.
\end{remark}

Similarly, given $b>0$, let $\kappa_b$ be the first entrance time of $Y$ into $(b,\infty)$ we know from Proposition 2 in \cite{Pistorius2004} that 
\begin{align} \label{upcrossing_time_reflected}
\e_x\Big[\expo^{-q\kappa_b}\Bigr]=\dfrac{Z^{(q)}(x)}{Z^{(q)}(b)}, \quad x \leq b.
\end{align}
In addition, we know from \cite[page 167]{AvPaPi07} that
\begin{align}
\mathbb{E}_x\biggr[\int_{[0,\kappa_b]}\expo^{-qt} \ud R^{0}_t\biggl]&=- k^{(q)}(x) +k(b)\dfrac{Z^{(q)}(x)} {Z^{(q)}(b)}, \quad x \leq b, \label{capital_injection_identity_SN}
\end{align}
where 
\begin{equation}\label{k.1}
k^{(q)}(x):= \overline{Z}^{(q)}(x)+ \dfrac {\psi'(0+)} {q}.
\end{equation}

\subsection{Optimal dividends without transaction costs and with capital injection}\label{dividends_injection_no_cost}

When $\delta=0$, \eqref{v.f.strategy} becomes
\begin{equation}\label{lm}
v_{\Lambda}^{\pi}(x)=\e_x\left[\int_0^{\infty}\expo^{-qt}\ud D^{\pi}_t-\Lambda\int_0^{\infty}\expo^{-qt}\ud R^{\pi}_t\right],
\end{equation}
for any initial capital $x\geq0$ and admissible policy $\pi=\{D^{\pi},R^{\pi}\}$.  Consider the constant barrier strategy $\pi_{a,0}=\{D^a,R^0\}$, where $D^a=(S-a)\vee0$ and $S$, $R^0$ are as in \eqref{reflected_levy}. The controlled risk process $X^{\pi_{a,0}}=X-D^a+R^0$ is a doubly reflected spectrally negative L\'evy process and was studied by Avram et al. in \cite{AvPaPi07}.  Using Theorem 1 in \cite{AvPaPi07} we have that for $a>0$ and $x\in[0,a]$, 
\begin{align}\label{eqdividendn}
\e_x\left[\int_0^{\infty}\expo^{-qt}\ud D^a_t\right]&=\dfrac{Z^{(q)}(x)}{qW^{(q)}(a)},\\\label{eqinjectionn}
\e_x\left[\int_0^{\infty}\expo^{-qt}\ud R^0_t\right]&=\dfrac{Z^{(q)}(a)}{qW^{(q)}(a)}Z^{(q)}(x)-k^{(q)}(x).
\end{align}
Note that the expression in \eqref{eqinjectionn} is finite by our assumption that $\psi'(0+)>-\infty$. Using the above expressions, we can see that for  $\Lambda\geq1$,
\begin{equation}\label{vf_1}
v_{\Lambda}^a(x):=v_{\Lambda}^{\pi_{a,0}}(x)=
\begin{cases}
Z^{(q)}(x)\zeta_{\Lambda}(a)+\Lambda k^{(q)}(x),&\text{if}\ 0\leq x\leq a,\\
x-a+v_{\Lambda}^{a}(a),&\text{if}\ x>a, 
\end{cases} 
\end{equation}
where 
\begin{equation}\label{zeta}
\zeta_{\Lambda}(a)=\dfrac{1-\Lambda Z^{(q)}(a)}{qW^{(q)}(a)},\ a>0.
\end{equation} 

Equation \eqref{vf_1} suggests that in order to find the best barrier strategy we should maximize the function $\zeta_{\Lambda}$. So we can define the candidate for optimal barrier by
\begin{equation}\label{a1}
a_{\Lambda}=\sup\{a\geq0: \zeta_{\Lambda}(a)\geq\zeta_{\Lambda}(x), \ \text{for all}\ x\geq0\}.
\end{equation}

\begin{remark}\label{remzeta} 
Note that $\zeta_{\Lambda}$ is a function from $(0,\infty)$ to $(-\infty,0)$ and satisfies 
\begin{equation*}
\lim_{a\rightarrow0}\zeta_{\Lambda}(a)=-\dfrac{\Lambda-1}{qW^{(q)}(0)}\ \quad\text{and}\quad\lim_{a\rightarrow\infty}\zeta_{\Lambda}(a)=-\dfrac{\Lambda}{\Phi(q)}.  
\end{equation*}
Where in the case $X$ is of unbounded variation the first equality is understood to be minus infinity.
The barrier level $a_{\Lambda}$, given in \eqref{a1}, agrees with the level defined in \cite{AvPaPi07}. Using the definition of the function $H$, we have that
\begin{equation}\label{derzeta}
\dfrac{\ud \zeta_{\Lambda}(a)}{\ud a}=\dfrac{\Lambda W^{(q)\prime}(a)}{q[W^{(q)}(a)]^2}( H(a)-1/\Lambda).
\end{equation}
Since $H$ is strictly decreasing, $\zeta_{\Lambda}$ has a unique maximum $a_{\Lambda}$, which is either a critical point, that is a solution of $H(a)=\dfrac{1}{\Lambda}$, or $0$ if the right-hand derivative of $\zeta_{\Lambda}$ is negative at $0$. Therefore, by Remark \ref{remH},
\begin{equation}\label{opt_an}
a_{\Lambda}=
\begin{cases}
0,&\text{if}\ \sigma=0,\ \Pi(0,\infty)<\infty\ \text{ and }\Lambda<1+\dfrac{q}{\Pi(0,\infty)},\\
H^{-1}\left(1/\Lambda\right),&\text{otherwise.} 
\end{cases} 
\end{equation}
Also, note that $\zeta_{\Lambda}$ is strictly increasing before $a_{\Lambda}$ and strictly decreasing after $a_\Lambda$.
\end{remark}

Hence, from \cite{AvPaPi07} we know that the value function \eqref{v.f} and the optimal strategy are given by $V_{\Lambda}=v_{\Lambda}^{a_{\Lambda}}$ and $\pi_{0,a_{\Lambda}}$, respectively, where $v_{\Lambda}^{a_{\Lambda}}$ is given by  \eqref{vf_1},  with $a=a_{\Lambda}$, respectively.  

\begin{remark}\label{rem_optbar_inf}

Note that the optimal barrier $a_{\Lambda}\rightarrow\infty$ as $\Lambda\rightarrow\infty$.
\end{remark}

\section{Capital injection and transaction costs}\label{neg.1}

In this section we will solve the problem \eqref{v.f} in the presence of transaction costs, i.e. $\delta>0$. We will consider strategies where the capital injection policy is $R^0$, which is given in \eqref{reflected_levy}, and the dividend strategy is the so-called reflected $(c_1,c_2)$-policies, which we define next. Recall that we defined the reflected L\'evy process from below at $0$, given by $Y=X+R^0$.

\subsection{Value function of reflected $(c_1,c_2)$-policies}

In this section we will define the reflected $(c_1,c_2)$-policy. To this end  let $c_2>c_1\geq0$ and $\{T_i^{c_1,c_2};i\geq 1\}$ be the set of stopping times defined as
\begin{align*}
T_i^{c_1,c_2}=\inf\{t\geq 0:Y_t>Y_0\vee c_2+(c_2-c_1)(i-1) \},\qquad i=1,2\dots
\end{align*}
Let $D^{c_1,c_2}=\{D^{c_1,c_2}_t:t\geq0\}$ be defined as
$$D^{c_1,c_2}_t=\mathbf{1}_{\{T_1^{c_1,c_2}<t\}}(Y_0\vee c_2-c_1)+\sum\limits_{i=2}^{\infty}\mathbf{1}_{\{T_i^{c_1,c_2}<t\}}(c_2-c_1),\quad t\geq0.$$
We will use the notation $\pi_{c_1,c_2}:=\{D^{c_1,c_2},R^0\}$ and denote the controlled process associated to this strategy by $X^{c_1,c_2}:=Y-D^{c_1,c_2}=X-D^{c_1,c_2}+R^0$. Let us compute the expected NPV of dividends with transaction costs. To this end, we denote
\begin{equation*}
f_{c_1,c_2}(x)=\mathbb{E}_x\left[\int_0^{\infty}\expo^{-qt}\diff\left(D^{c_1,c_2}_t-\delta\sum_{0\leq s<t}\mathbf{1}_{\{\Delta D^{c_1,c_2}_s>0\}}\right)\right].
\end{equation*}
%\blue{[La notaci\'on $v_{c_1,c_2}$ puede ser confusa. que tal $f_D$ o algo as\'i? Es solo para esta secci\'on.]}
If $x<c_2$ then, by the Strong Markov Property and \eqref{upcrossing_time_reflected}, we obtain that
\begin{align}\label{1}
{f}_{c_1,c_2}(x)=\mathbb{E}_x\Bigr[\expo^{-q T_1^{c_1,c_2}}  \Bigl] {f}_{c_1,c_2}(c_2)=\dfrac{Z^{(q)}(x)}{Z^{(q)}(c_2)}{f}_{c_1,c_2}(c_2).
\end{align}
When $x\geq c_2$ a dividend of amount $c_2-c_1$ is paid immediately plus a transaction cost of $\delta$, so by using \eqref{1} we obtain
\begin{align*}
f_{c_1,c_2}(x)=x-c_1-\delta+f_{c_1,c_2}(c_1)=x-c_1-\delta+\dfrac{Z^{(q)}(c_1)}{Z^{(q)}(c_2)}f_{c_1,c_2}(c_2).
\end{align*}
Hence taking $x=c_2$, and solving for $v_{c_1,c_2}(c_2)$ we get
\[{f}_{c_1,c_2}(c_2)=(c_2-c_1-\delta)\dfrac{Z^{(q)}(c_2)}{Z^{(q)}(c_2)-Z^{(q)}(c_1)}.
\]
Using the above expression in \eqref{1} we  have for $x<c_2$,
\begin{equation}\label{div_trans_c_2>x}
{f_{c_1,c_2}}(x)=(c_2-c_1-\delta)\dfrac{Z^{(q)}(x)}{Z^{(q)}(c_2)-Z^{(q)}(c_1)}.
\end{equation}
Now, let us compute the expected net present value of the injected capital denoted by
\[
{g_{c_1,c_2}}(x)=\mathbb{E}_x\left[\int_0^{\infty}\expo^{-qt}\diff R^0_{t}\right].
\]
%\blue{[Aqui lo mismo, que tal $f_R$ o algo as\'i?]}
Again, by the Strong Markov Property, noting that $T_{1}^{c_{1},c_{2}}=\inf\{t>0: Y_{t}\in(c_{2},\infty)\}$ and from \eqref{upcrossing_time_reflected}--\eqref{capital_injection_identity_SN}, we have for $x\geq0$
\begin{align*}
g_{c_1,c_2}(x)&=\mathbb{E}_x\left[\int_{[0,T_1^{c_1,c_2}]}\expo^{-qt}\diff R^0_{t}\right]+\mathbb{E}_x\left[\expo^{-qT_1^{c_1,c_2}} \right] g_{c_1,c_2}(c_1)\\
&=-k^{(q)}(x)+k^{(q)}(c_{2})\dfrac{Z^{(q)}(x)}{Z^{(q)}(c_{2})}+\dfrac{Z^{(q)}(x)} {Z^{(q)}(c_2)} g_{c_1,c_2}(c_1),
\end{align*}
So, setting $x=c_1$ and solving for $v_R(c_1)$ we obtain
\begin{align*}
g_{c_1,c_2}(c_1)&=\bigg(-k^{(q)}(c_{1})+k^{(q)}(c_{2})\dfrac{Z^{(q)}(c_{1})}{Z^{(q)}(c_{2})}\bigg)\dfrac{Z^{(q)}(c_2)}{Z^{(q)}(c_2)-Z^{(q)}(c_1)}\\
&=\left(-\overline{Z}^{(q)}(c_1)+\overline{Z}^{(q)}(c_2)\dfrac{Z^{(q)}(c_1)}{Z^{(q)}(c_2)}\right)\dfrac{Z^{(q)}(c_2)}{Z^{(q)}(c_2)-Z^{(q)}(c_1)}- \dfrac {\psi'(0+)}{q}.
\end{align*}
Putting the pieces together we obtain
\begin{align*}
g_{c_1,c_2}(x)&=-k^{(q)}(x) +k^{(q)}(c_{2}) \dfrac{Z^{(q)}(x)} {Z^{(q)}(c_2)}\\
&\quad+\left(\left(-\overline{Z}^{(q)}(c_1)+\overline{Z}^{(q)}(c_2)\dfrac{Z^{(q)}(c_1)}{Z^{(q)}(c_2)}\right)\dfrac{Z^{(q)}(c_2)}{Z^{(q)}(c_2)-Z^{(q)}(c_1)}- \dfrac {\psi'(0+)}{q}\right)\dfrac{Z^{(q)}(x)}{Z^{(q)}(c_2)}\\
&=-k^{(q)}(x) +\overline{Z}^{(q)}(c_2)\dfrac{Z^{(q)}(x)} {Z^{(q)}(c_2)}+\left(- \overline{Z}^{(q)}(c_1)  +\overline{Z}^{(q)}(c_2) \dfrac{Z^{(q)}(c_1)} {Z^{(q)}(c_2)}\right)\dfrac{Z^{(q)}(x)}{Z^{(q)}(c_2)-Z^{(q)}(c_1)}\\
&=Z^{(q)}(x)\bigg(\dfrac{\overline{Z}^{(q)}(c_{2})-\overline{Z}^{(q)}(c_{1})}{Z^{(q)}(c_{2})-Z^{(q)}(c_{1})}\bigg)- k^{(q)}(x).
\end{align*}
Hence we have the following result.
\begin{lemma}\label{P1}
The expected NPV associated with a reflected $(c_1,c_2)$-policy is given by
\begin{equation}\label{vf}
v_{\Lambda}^{c_1,c_2}(x):=v_{\Lambda}^{\pi_{c_1,c_2}}(x)=
\begin{cases}
Z^{(q)}(x)G_{\Lambda}(c_1,c_2)+\Lambda k^{(q)}(x), &\text{if}\  x\leq c_{2},\\
x-c_1-\delta+v_{\Lambda}^{c_1,c_2}(c_1), & \text{if}\ x> c_{2}.
\end{cases}
\end{equation}
where
\begin{equation}\label{funG}
G_{\Lambda}(c_1,c_2):=\dfrac{c_2-c_1-\delta-\Lambda\left(\overline{Z}^{(q)}(c_2)-\overline{Z}^{(q)}(c_1)\right)}{Z^{(q)}(c_2)-Z^{(q)}(c_1)},\quad \text{for all }c_2>c_1\geq0.
\end{equation}
\end{lemma}

\begin{remark}\label{limG}
Note that $G_{\Lambda}$ is $\hol^{2}$ on $\mathcal{A}:=\{(c_{1},c_{2})\in\R^{2}_{+}: c_{1}<c_{2}\}$, and 
\begin{equation*}
\lim_{c_2\downarrow c_1}G_{\Lambda}(c_{1},c_{2})=-\infty\quad\text{and}\quad\lim_{|c_{1}|+|c_{2}|\rightarrow\infty}G_{\Lambda}(c_{1},c_{2})=-\dfrac{\Lambda}{\Phi(q)}.
\end{equation*}
\end{remark}	
\subsection{Choice of optimal thresholds}

In order to choose the optimal thresholds among reflected policies we will maximize the function $G_{\Lambda}$.

\begin{proposition}\label{ex_max}
The function $G_{\Lambda}$, defined in \eqref{funG}, attains its maximum on $\mathcal{A}$.
\end{proposition}
\begin{proof}
Let $c_1\geq0$ be fixed. The first derivative of $G_{\Lambda}$ with respect to $c_2$ is given by
\begin{align}\label{m5}
\partial_{ c_{2}}G_{\Lambda}(c_1,c_2)=\dfrac{qF_{\Lambda}(c_1,c_2)W^{(q)}(c_2)}{(Z^{(q)}(c_2)-Z^{(q)}(c_1))^2},
\end{align}
where
\begin{align}\label{m1}
F_{\Lambda}(c_1,c_2):=&\dfrac{(Z^{(q)}(c_2)-Z^{(q)}(c_1))}{qW^{(q)}(c_2)}(1-\Lambda Z^{(q)}(c_2))-\left(c_2-c_1-\delta-\Lambda\left(\overline{Z}^{(q)}(c_2)-\overline{Z}^{(q)}(c_1)\right)\right)\\
=&-\Lambda\left[\dfrac{(Z^{(q)}(c_2))^2}{qW^{(q)}(c_2)}-\overline{Z}^{(q)}(c_2)-\left(\dfrac{(Z^{(q)}(c_2))}{qW^{(q)}(c_2)}Z^{(q)}(c_1)-\overline{Z}^{(q)}(c_1)\right)\right]\notag\\
&+\dfrac{Z^{(q)}(c_2)-Z^{(q)}(c_1)}{qW^{(q)}(c_2)}-(c_2-c_1-\delta).\notag
\end{align}
On the other hand, taking $a=c_{2}$ in \eqref{eqinjectionn} we see   
\begin{equation*}
\dfrac{[Z^{(q)}(c_{2})]^{2}}{qW^{(q)}(c_{2})} -k^{(q)}(c_{2})\geq0\quad\text{and}\quad\dfrac{Z^{(q)}(c_{2})}{qW^{(q)}(c_{2})}Z^{(q)}(c_{1}) -k^{(q)}(c_{1})\geq0.
\end{equation*}
Then, using \eqref{k.1}, we have
\begin{equation}\label{fun_F}
F_{\Lambda}(c_{1},c_{2})<\dfrac{Z^{(q)}(c_{2})}{qW^{(q)}(c_{2})}+\Lambda\biggr[\dfrac{Z^{(q)}(c_{2})}{qW^{(q)}(c_{2})}Z^{(q)}(c_{1})-k^{(q)}(c_{1})\biggl]-(c_{2}-c_{1}-\delta).
\end{equation}
Therefore, since  $\lim\limits_{c_2\rightarrow\infty}\dfrac{Z^{(q)}(c_2)}{qW^{(q)}(c_2)}=\dfrac{1}{\Phi(q)}$ (see Remark \ref{remark_smoothness_zero}), the right-hand side of the above inequality goes to $-\infty$ as $c_2$ goes to $\infty$, which implies
\begin{equation}\label{m2}
\partial_{ c_{2}}G_{\Lambda}(c_1,c_2)<0\ \text{for}\ c_2\ \text{large}\ \text{enough}. 
\end{equation}
From here and Remark \ref{limG}, we obtain that there exists $c^{*}\in(c_{1},\infty)$ (that depends on $c_{1}$) such that
\begin{equation}\label{m3}
G_{\Lambda}(c_{1},c_{2})\leq G_{\Lambda}(c_{1},c^{*})\ \text{for all}\ c_{2}>c_{1}.
\end{equation}
Taking $d^{*}(c_{1}):=\sup\{c^{*}>c_{1}: G_{\Lambda}(c_{1},c_{2})\leq G_{\Lambda}(c_{1},c^{*})\ \text{for all}\ c_{2}>c_{1} \}$, with $c_{1}\geq0$, we see $d^{*}(c_{1})<\infty$ for each $c_{1}\geq0$, since \eqref{m2} holds. From \eqref{m5} and that $\partial_{c_{2}}G_{\Lambda}(c_{1},d^{*}(c_{1}))=0$, it follows $F_{\Lambda}(c_{1},d^{*}(c_{1}))=0$ for $c_{1}\geq0$. Then, by definition of $F_{\Lambda}$ and $\zeta_{\Lambda}$; see \eqref{m1} and \eqref{zeta} respectively, we get 
\begin{equation}\label{m4}
G_{\Lambda}(c_1,d^*(c_1))=\dfrac{d^{*}(c_{1})-c_{1}-\delta-\Lambda(\overline{Z}^{(q)}(d^{*}(c_{1}))-\overline{Z}^{(q)}(c_{1}))}{Z^{(q)}(d^{*}(c_{1}))-Z^{(q)}(c_{1})}=\zeta_{\Lambda}(d^{*}(c_{1})),
\end{equation}
for each $c_{1}\geq0$.  Now let us take $\bar{c}_1>a_{\Lambda}$ (where $a_{\Lambda}$ is defined in \eqref{zeta}), then using the fact that $\zeta_{\Lambda}$ is strictly decreasing in $(a_{\Lambda},\infty)$ (see Remark \ref{remzeta}) we have that for any $c_2>c_1>d^*(\bar{c}_1)$ it holds that $d^*(\bar{c}_1)<d^*(c_1)$ and
	\begin{align*}
		G_\Lambda(c_1,c_2)\leq G_{\Lambda}(c_1,d^*(c_1))=\zeta_{\Lambda}(d^{*}(c_{1}))<\zeta_{\Lambda}(d^{*}(\bar{c}_{1}))=G_{\Lambda}(\bar{c}_1,d^*(\bar{c}_1)).
	\end{align*}
This implies that the maximum of the function $G_{\Lambda}$ has to be achieved in the set $\{(c_1,c_2)\in\mathbb{R}^2_+:c_1<c_2\ \text{and}\ c_1\in[0,\bar{c}_1]\}$.
Finally from \eqref{fun_F} we obtain for $c_1\in[0,\bar{c}_1]$ 
\begin{align*}
F_{\Lambda}(c_{1},c_{2})<\dfrac{Z^{(q)}(c_{2})}{qW^{(q)}(c_{2})}+\Lambda\sup_{c_1\in[0,\bar{c}_1]}\biggr[\dfrac{Z^{(q)}(c_{2})}{qW^{(q)}(c_{2})}Z^{(q)}(c_{1})-k^{(q)}(c_{1})\biggl]-(c_{2}-\bar{c}_{1}-\delta).
\end{align*}
Hence, for any $c_1\in[0,\bar{c}_1]$ we can find $\bar{c_2}>\bar{c}_1$ such that 
\[
\partial_{ c_{2}}G_{\Lambda}(c_1,c_2)(c_1,c_2)<0\qquad\text{for any $0\leq c_1\leq \bar{c_1}$ and $0\leq c_2\leq \bar{c}_2$.}
\] 
Therefore, the function $G_{\Lambda}$ attains its maximum in the set $\{(c_{1},c_{2})\in[0,\bar{c}_1]\times[0,\bar{c}_2]: c_{1}<c_{2}\}\subset\mathcal{A}$.
\end{proof}
Note that by Proposition \ref{ex_max} the set $\mathcal{B}\subset\mathcal{A}$ defined as
\begin{equation}
\mathcal{B}:=\{(c^{*}_{1},c^{*}_{2})\in\mathcal{A}:G_{\Lambda}(c^{*}_{1},c^{*}_{2})\geq G_{\Lambda}(c_{1},c_{2})\ \text{for all}\ (c_{1},c_{2})\in\mathcal{A}\},
\end{equation} 
is not empty. Moreover, since $G_{\Lambda}\in \hol^{1}(\mathcal{A})$ an using \eqref{zeta}, it follows that
\begin{equation}\label{con_1}
\partial_{c_1}G_{\Lambda}(c^{*}_1,c^{*}_2)=\dfrac{qW^{(q)}(c^{*}_1)}{Z^{(q)}(c^{*}_2)-Z^{(q)}(c^{*}_1)}\left(G_{\Lambda}(c^{*}_1,c^{*}_2)-\zeta_{\Lambda}(c^{*}_1)\right)\leq0,\ \text{for}\ (c^{*}_{1},c^{*}_{2})\in\mathcal{B},
\end{equation}
with equality if $c_1>0$, and
\begin{equation}\label{con_2}
\partial_{c_2}G_{\Lambda}(c^{*}_1,c^{*}_2)=-\dfrac{qW^{(q)}(c^{*}_2)}{Z^{(q)}(c^{*}_2)-Z^{(q)}(c^{*}_1)}\left(G_{\Lambda}(c^{*}_1,c^{*}_2)-\zeta_{\Lambda}(c^{*}_2)\right)=0,\ \text{for}\ (c^{*}_{1},c^{*}_{2})\in\mathcal{B}.
\end{equation}

\begin{proposition}\label{remoptc}
	There exists a unique pair $(c_{1}^{\Lambda},c_{2}^{\Lambda})$ in $\mathcal{B}$. Furthermore,  $0\leq c_{1}^{\Lambda}\leq a_{\Lambda}< c_2^{\Lambda}<\infty$, with $a_{\Lambda}$ defined in \eqref{opt_an}, and the value function associated with the $(c^{\Lambda}_{1},c^{\Lambda}_{2})$-policy is
	\begin{equation}\label{vf_2}
	v_{\Lambda}^{c_1^{\Lambda},c_2^{\Lambda}}(x)=
	\begin{cases}
	Z^{(q)}(x)\zeta_{\Lambda}(c_2^{\Lambda})+ \Lambda k^{(q)}(x),& \text{if}\ x\leq c_{2}^{\Lambda},\\
	x-c_2^{\Lambda}+v_{\Lambda}^{c_1^{\Lambda},c_2^{\Lambda}}(c_2^{\Lambda}),&\text{if}\ x>c_{2}^{\Lambda}.
	\end{cases}
	\end{equation}
\end{proposition}	

\begin{proof}
	Let $M$ be the maximum value of $G_{\Lambda}$ in $\mathcal{B}$, therefore, for any $(c_1^*,c_2^*)\in\mathcal{B}$ we have that $\zeta_{\Lambda}(c_2^*)=M$ by \eqref{con_2}. From Remark \ref{remzeta} we know that $\zeta_{\Lambda}$ is strictly increasing before $a_{\Lambda}$ and strictly decreasing after $a_\Lambda$.  If $\zeta_\Lambda(0)\geq M$, $\zeta_\Lambda$ attains $M$ at a unique $c_2^\Lambda>a_\Lambda$ and therefore $(0,c_2^\Lambda)$ is the only point that satisfies \eqref{con_1}. On the other hand, if $\zeta_\Lambda(0)< M$, $\zeta_\Lambda$ can only attain the value $M$ at a unique $c_1^\Lambda<a_\Lambda$ and a unique $c_2^\Lambda>a_\Lambda$. Hence, $(c_{1}^{\Lambda},c_{2}^{\Lambda})$ is the only point that satisfies \eqref{con_1} and \eqref{con_2}, that is, the only point in $\mathcal{B}$. Now, from Lemma \ref{P1} and using that $G_{\Lambda}(c^{\Lambda}_1,c^{\Lambda}_2)=\zeta_{\Lambda}(c^{\Lambda}_2)$, we obtain the first part of \eqref{vf_2}. For the second part, let $x>c_{2}^{\Lambda}$ so,
	\begin{align*}%\label{vf_2_b}
		v_{\Lambda}^{c_1^{\Lambda},c_2^{\Lambda}}(x)&=x-c_1^{\Lambda}-\delta+v_{\Lambda}^{c_1^{\Lambda},c_2^{\Lambda}}(c_1^{\Lambda})\\
		&=x-c_2^{\Lambda}+c_2^{\Lambda}-c_1^{\Lambda}-\delta+v_{\Lambda}^{c_1^{\Lambda},c_2^{\Lambda}}(c_1^{\Lambda})\\
		&=x-c_2^{\Lambda}+v_{\Lambda}^{c_1^{\Lambda},c_2^{\Lambda}}(c_2^{\Lambda}).
	\end{align*}
\end{proof}

The following properties of $v_{\Lambda}^{c_1^{\Lambda},c_2^{\Lambda}}$ will be used later in the verification thereom.
\begin{remark}\label{lower_bound}
	From \eqref{k.1} and \eqref{vf_2}, we note 
	\begin{align*}
		v_{\Lambda}^{c_1^{\Lambda},c_2^{\Lambda}}(x)\geq \dfrac{\Lambda\psi'(0+)}{q}+Z^{(q)}(c_2^{\Lambda})\zeta(c_2^{\Lambda}),\ \text{for $x>0$}.
	\end{align*}
\end{remark}

\begin{remark}[Continuity/smoothness at zero]\label{cont_vf}
	Note that for $x<0$, $v_{\Lambda}^{c_1^{\Lambda},c_2^{\Lambda}}(x)=v_{\Lambda}^{c_1^{\Lambda},c_2^{\Lambda}}(0)+\Lambda x$. Therefore,
	\begin{itemize} 
		\item[(i)] $v_{\Lambda}^{c_1^{\Lambda},c_2^{\Lambda}}$ is continuous at zero.
		\item[(ii)] For the case of unbounded variation we have that \[v_{\Lambda}^{c_1^{\Lambda},c_2^{\Lambda}\prime}(0+)=qW^{(q)}(0+)\zeta_{\Lambda}(c_2^{\Lambda})+\Lambda=\Lambda= v_{\Lambda}^{c_1^{\Lambda},c_2^{\Lambda}\prime}(0-).\]
	\end{itemize}
\end{remark}

\subsection{Verification}
Let us denote by $v_{\Lambda}$ the function defined in \eqref{vf_2}, which is the optimal value function among reflected policies. We now prove some properties of this function.

\begin{lemma}\label{smoothness}
	The function $v_{\Lambda}$ is  $C^{2}((0,\infty)\backslash\{c_2^{\Lambda}\})$  and  $C^{1}(0,\infty)$.
\end{lemma}

\begin{proof}
	By Assumption \ref{assump_C1}, we have that for each $q\geq0$, the function $W^{(q)}$ is continuously differentiable in $(0,\infty)$. This implies, by \eqref{vf_2}, that $v_{\Lambda}$ is $C^{2}((0,\infty)\backslash\{c_2^{\Lambda}\})$. On the other hand using \eqref{vf_2} we have that for $x\leq c_2^{\Lambda}$ we obtain
	\begin{align*}
		v_{\Lambda}^{\prime}(x)&=qW^{(q)}(x)\zeta_{\Lambda}(c_2^{\Lambda})+ \Lambda Z^{(q)}(x)\\
		&=qW^{(q)}(x)\left(\dfrac{1-\Lambda Z^{(q)}(c_2^{\Lambda})}{qW^{(q)}(c_2^{\Lambda})}\right)+\Lambda Z^{(q)}(x).
	\end{align*}
	This implies that $v_{\Lambda}^{\prime}(c_2^{\Lambda}-)=1$. On the other hand for $x>c_2^{\Lambda}$ we obtain by \eqref{vf_2} that
	\[
	v_{\Lambda}^{\prime}(c_2^{\Lambda}+)=1=v_{\Lambda}^{\prime}(c_2^{\Lambda}-),
	\]
	which implies the result.
\end{proof}

We let $\mathcal{L}$ be the operator, defined by
\begin{equation*}
	\mathcal{L} F(x) := \gamma F'(x)+\dfrac{\sigma^2}{2}F''(x) +\int_{(0,\infty)}(F(x-z)-F(x)+F'(x)z\mathbf{1}_{\{0<z\leq1\}})\Pi(\mathrm{d}z), \quad x > 0,
\end{equation*}
where $x\in\mathbb{R}$ and $F$ is a function on $\mathbb{R}$ such that $\mathcal{L}F(x)$ is well defined.

\begin{proposition}\label{generator_HJB}
	\begin{enumerate}
		\item $(\mathcal{L}-q)v_{\Lambda}(x)=0$ for $x<c_2^{\Lambda}$.
		\item $(\mathcal{L}-q)v_{\Lambda}(x)\leq0$ for $x>c_2^{\Lambda}$.
	\end{enumerate}
\end{proposition}

\begin{proof}
	\begin{enumerate}
		\item By the proof of Theorem 2.1 in \cite{KyYa14} we have that for $0<x<c_{2}^{\Lambda}$,
		\begin{align*}
			(\mathcal{L}-q)\Big( \overline{Z}^{(q)}(x) + \dfrac {\psi'(0+)}{q}\Big)=0\quad\text{and }\quad(\mathcal{L}-q)Z^{(q)}(x)=0.
		\end{align*}
		This implies that for $0<x<c_2^{\Lambda}$, 
		\begin{equation*}%\label{gen_1}
			(\mathcal{L}-q)v_{\Lambda}(x)=0.
		\end{equation*}
		\item We note that $v_{\Lambda}(y)=u_{\Lambda}^{c_2^{\Lambda}}(y)$ for all $y\geq0$, where $u_{\Lambda}^{a}$ is the barrier strategy at level $a$ for the dividend problem with capital injection given by \eqref{vf_1}. Therefore,
		\begin{itemize}
			\item [(i)] If we take $y\leq x$, and $c_2^{\Lambda}\leq x$, we obtain
			\[
			u_{\Lambda}^{x}(y)=Z^{(q)}(y)\zeta_{\Lambda}(x)+\Lambda\Big( \overline{Z}^{(q)}(y) + \dfrac {\psi'(0+)}{q}\Big).
			\]
			Recall the functions $\zeta_{\Lambda}$ and $H$ is as in \eqref{zeta} and \eqref{defH} respectively. Then
			\begin{align*}
				\lim_{y\uparrow x}\dfrac{\diff^{2} u_{\Lambda}^{x}}{\diff y^{2}}(y)&=\Lambda qW^{(q)}(x)+qW^{(q)\prime}(x)\zeta_{\Lambda}(x)\\
				&=\dfrac{W^{(q)\prime}(x)}{W^{(q)}(x)}\left(1-\Lambda\left(Z^{(q)}(x)-qW^{(q)}(x)\dfrac{W^{(q)}(x)}{W^{(q)\prime}(x)}\right)\right)\\
				&=\dfrac{W^{(q)\prime}(x)}{W^{(q)}(x)}(1-\Lambda H(x))\\
				&=-qW^{(q)}(x)\zeta'_{\Lambda}(x),
			\end{align*}
			By Proposition \ref{remoptc}, we know that $a_{\Lambda}<c_{2}^{\Lambda}\leq x$. Then, $\displaystyle\lim_{y\uparrow x}\dfrac{\diff^{2} u_{\Lambda}^{x}}{\diff y^{2}}(y)\geq0=\dfrac{\diff^{2} u_{\Lambda}^{c_{2}^{\Lambda}}}{\diff x^{2}}(x)$, %\red{no es mejor dejarlo con $\frac{\diff^{2}}{\diff y^{2}}$?} 
			since $\zeta'_{\Lambda}(x) <0$ by Remark \ref{remzeta}.
			\item[(ii)]
			We have for $y\in[0,c_2^{\Lambda}]$ that
			\begin{align*}
				\dfrac{\diff u_{\Lambda}^{c_2^{\Lambda}}}{\diff y}(y)=\Lambda Z^{(q)}(y)+qW^{(q)}(y)\zeta_{\Lambda}(c_2^{\Lambda})\geq \Lambda Z^{(q)}(y)+qW^{(q)}(y)\zeta_{\Lambda}(x)=\dfrac{\diff u_{\Lambda}^{x}}{\diff y}(y),
			\end{align*}
			which follows from the fact that for $x\geq c_2^{\Lambda}>a_{\Lambda}$, then $\zeta_{\Lambda}(c_2^{\Lambda})\geq \zeta_{\Lambda}(x)$ by Remark \ref{remoptc}. On the other hand for $y\in(c_2^{\Lambda},x]$ we have, using the fact that $\zeta_{\Lambda}(y)\geq \zeta_{\Lambda}(x)$,
			\begin{align*}
				\dfrac{\diff u_{\Lambda}^{x}}{\diff y}(y)&=\Lambda Z^{(q)}(y)+qW^{(q)}(y)\zeta_{\Lambda}(x)\\
				&\leq \Lambda Z^{(q)}(y)+qW^{(q)}(y)\zeta_{\Lambda}(y)\\
				&=\Lambda Z^{(q)}(y)+qW^{(q)}(y)\dfrac{1-\Lambda Z^{(q)}(y)}{qW^{(q)}(y)}=1=	\dfrac{\diff u_{\Lambda}^{c_2^{\Lambda}}}{\diff y}(y).%\ \text{\red{No es $	\dfrac{\diff u_{\Lambda}^{c_2^{\Lambda}}}{\diff y}(c^{\Lambda}_{2})$?}}
			\end{align*}
			\item[(iii)] We note that
			\begin{align*}
				u_{\Lambda}^x(c_2^{\Lambda})&=\Lambda\Big( \overline{Z}^{(q)}(c_2^{\Lambda}) + \dfrac {\psi'(0+)}{q}\Big)+Z^{(q)}(c_2^{\Lambda})\zeta_{\Lambda}(x)\\
				&\leq \Lambda\Big( \overline{Z}^{(q)}(c_2^{\Lambda}) + \dfrac {\psi'(0+)}{q}\Big)+Z^{(q)}(c_2^{\Lambda})\zeta_{\Lambda}(c_2^{\Lambda})=u_{\Lambda}^{c_2^{\Lambda}}(c_2^{\Lambda}).
			\end{align*}
			This and (ii) imply that $(u_{\Lambda}^{c_2^{\Lambda}}-u_{\Lambda}^x)(x)\geq 0$.
			\item[(iv)] We have
			\[
			\dfrac{\diff u_{\Lambda}^{c_2^{\Lambda}}}{\diff x}(x)=1=\lim_{y\to x}\dfrac{\diff u_{\Lambda}^{x}}{\diff y}(y).
			\]
		\end{itemize}
		So mimicking the proof of Theorem 2 in \cite{Loeffen082} we obtain the result.
	\end{enumerate}
\end{proof}

\begin{lemma}\label{P2}
	\begin{enumerate}
		\item For $x>0$ we have that $v_{\Lambda}^{\prime}(x)\leq \Lambda$.
		\item For $x\geq y\geq 0$ we have that $v_{\Lambda}(x)-v_{\Lambda}(y)\geq x-y-\delta$.
	\end{enumerate}
\end{lemma}

\begin{proof}
	\begin{enumerate}
		\item By \eqref{laplace_in_terms_of_z} together with \eqref{vf_2}, we note that for $x\leq c_{2}^{\Lambda}$,
		\begin{align*}
			v_{\Lambda}^{\prime}(x)
			&=\Lambda \left( Z^{(q)}(x)-\dfrac{W^{(q)}(x)}{W^{(q)}(c_{2}^{\Lambda})}Z^{(q)}(c_{2}^{\Lambda})\right)+\dfrac{W^{(q)}(x)}{W^{(q)}(c_{2}^{\Lambda})}\\
			&=\Lambda \E_x\left[e^{-q\tau_0^-}\mathbf{1}_{\left\{\tau_0^-<\tau_{c_{2}^{\Lambda}}^+\right\}}\right]+\E_x\left[e^{-q\tau_{c_{2}^{\Lambda}}^+}\mathbf{1}_{\left\{\tau_{c_{2}^{\Lambda}}^+<\tau_0^-\right\}}\right]\\
			&\leq\Lambda \left(\E_x\left[e^{-q\tau_0^-}\mathbf{1}_{\left\{\tau_0^-<\tau_{c_{2}^{\Lambda}}^+\right\}}\right]+\E_x\left[e^{-q\tau_{c_{2}^{\Lambda}}^+}\mathbf{1}_{\left\{\tau_{c_{2}^{\Lambda}}^+<\tau_0^-\right\}}\right]\right)\\
			&\leq\Lambda \left(\mathbb{P}_x\left[\tau_0^-<\tau_{c_{2}^{\Lambda}}^+\right]+\mathbb{P}_x\left[\tau_{c_{2}^{\Lambda}}^+<\tau_0^-\right]\right)= \Lambda .
		\end{align*}
		On the other hand, $v_{\Lambda}^{\prime}(x)=1\leq\Lambda$ for $x>c_{2}^{\Lambda}$. 
		\item Let us consider $c_{2}^{\Lambda}\geq x\geq y$ we note that
		\begin{align}\label{diff_1}
			v_{\Lambda}(x)-v_{\Lambda}(y)&=\Lambda \left(\overline{Z}^{(q)}(x)-\overline{Z}^{(q)}(y)\right)+\left(Z^{(q)}(x)-Z^{(q)}(y)\right)\zeta_{\Lambda }(c_{2}^{\Lambda})\notag\\
			&=\Lambda \left(\overline{Z}^{(q)}(x)-\overline{Z}^{(q)}(y)\right)+(Z^{(q)}(x)-Z^{(q)}(y))G_{\Lambda }(c_{1}^{\Lambda},c_{2}^{\Lambda})\notag\\
			&\geq\Lambda \left(\overline{Z}^{(q)}(x)-\overline{Z}^{(q)}(y)\right)+(Z^{(q)}(x)-Z^{(q)}(y))G_{\Lambda }(x,y)\notag\\
			&=\Lambda \left(\overline{Z}^{(q)}(x)-\overline{Z}^{(q)}(y)\right)+(Z^{(q)}(x)-Z^{(q)}(y))\dfrac{x-y-\delta-\Lambda \left(\overline{Z}^{(q)}(x)-\overline{Z}^{(q)}(y)\right)}{Z^{(q)}(x)-Z^{(q)}(y)}\notag\\
			&=x-y-\delta.
		\end{align}
		Now suppose that $x\geq y\geq c_{2}^{\Lambda}$, then using \eqref{vf_2} we obtain
		\begin{align*}
			v_{\Lambda}(x)-v_{\Lambda}(y)=x-y\geq x-y-\delta.
		\end{align*}
		Finally, for the case $x\geq c_{2}^{\Lambda}\geq y$ we have by \eqref{diff_1}
		\begin{equation*}
			v_{\Lambda}(x)-v_{\Lambda}(y)=x-c_{2}^{\Lambda}+v_{\Lambda}(c_{2}^{\Lambda})-v_{\Lambda}(y)\geq x-c_{2}^{\Lambda}+(c_{2}^{\Lambda}-y-\delta)=x-y-\delta.
		\end{equation*}
	\end{enumerate}
\end{proof}

We are now ready for the Verification Theorem that proves the optimality of the $(c_{1}^{\Lambda},c_{2}^{\Lambda})$-policy.
\begin{theorem}[Verification Theorem]\label{L.V.1}
	We have that $v_{\Lambda}=V_{\Lambda}(x)$ for all $x\geq 0$. Hence the $(c_{1}^{\Lambda},c_{2}^{\Lambda})$-policy is optimal.
\end{theorem}
\begin{proof}
	By the definition of $V_{\Lambda}$ as a supremum, it follows that $v_{\Lambda}(x)\leq V_{\Lambda}(x)$ for all $x\geq0$. Let us show that $v_{\Lambda}(x)\geq v^\pi_\Lambda(x)$ for all admissible $\pi\in\Theta$ and for all $x\geq0$. Fix $\pi=\{D^\pi,R^\pi\}$ and let $(T_n)_{n\in\mathbb{N}}$ be the sequence of stopping times defined by $T_n :=\inf\{t>0:{X}^\pi_t>n \}$.  Since ${X}^\pi$ is a semi-martingale and $v_{\Lambda}$ is sufficiently smooth on $(0, \infty)$ by Lemma \ref{smoothness}, and continuous (resp. continuously differentiable) at zero for the case of bounded variation (resp. unbounded variation) by Remark \ref{cont_vf}, we can use the change of variables/Meyer-It\^o's formula (cf.\ Theorems II.31 and II.32 of \cite{protter})  to the stopped process $(e^{-q(t\wedge T_n)}v_\Lambda({X}^\pi_{t\wedge T_n}); t \geq 0)$ to deduce under $\mathbb{P}_x$ that
	\begin{align}\label{impulse_verif_1}
		e^{-q(t\wedge T_n)}v_{\Lambda}({X}^\pi_{t\wedge T_n})-v_{\Lambda}(x)= & \int_{0}^{t\wedge T_n}e^{-qs} (\mathcal{L}-q)v_{\Lambda}({X}^\pi_{s-}) \mathrm{d}s+M_{t\wedge T_n}+J_{t\wedge T_n}\\ \nonumber
		&+\int_{[0, t\wedge T_n]}e^{-qs}v_{\Lambda}^{\prime}({X}^\pi_{s-}) \mathrm{d} R_s^{\pi,c},
	\end{align}
	where $M$ is a local martingale with $M_0=0$, $R^{\pi,c}$ and $D^{\pi,c}$ are the continuous parts of $R^{\pi}$ and $D^{\pi}$ respectively, and  $J$ is given by
	\begin{align*}
		J_t&=\sum_{s\leq t}e^{-qs}\left[v_{\Lambda}({X}^\pi_{s-}+\Delta[X_{s}+R^{\pi}_{s}])-v_{\Lambda}({X}^\pi_{s-}+\Delta  X_{s})\right]1_{\{\Delta[X_{s}+R^{\pi}_{s}]_s\not=0\}}\\
		&+\sum_{s\leq t}e^{-qs}e^{-qs}[v_{\Lambda}({X}^\pi_{s-}+\Delta [X_{s}+R^{\pi}_{s}]-\Delta D_s^\pi)-v_{\Lambda}({X}^\pi_{s-}+\Delta [X_{s}+R^{\pi}_{s}]))1_{\{\Delta D_s^\pi\not=0\}}\qquad\text{ for $t\geq0$.}
	\end{align*}
	On the other hand by part (1) of Lemma \ref{P2}, we obtain that
	\begin{align*}
		&\int_{[0, t\wedge T_n]}e^{-qs}v_{\Lambda}^{\prime}({X}^\pi_{s-}) \mathrm{d} R_s^{\pi,c}\\&+\sum_{0 \leq s\leq t\wedge T_n}e^{-qs}\left[v_{\Lambda}({X}^\pi_{s-}+\Delta[X_{s}+R^{\pi}_{s}])-v_{\Lambda}({X}^\pi_{s-}+\Delta  X_{s})\right]1_{\{\Delta[X_{s}+R^{\pi}_{s}]_s\not=0\}}\\
		&\leq \Lambda\int_{[0, t\wedge T_n]}e^{-qs} \mathrm{d} R_s^{\pi,c} +\Lambda\sum_{0 \leq s\leq t\wedge T_n} e^{-qs} \Delta R_s^{\pi}=\Lambda\int_{[0, t\wedge T_n]}e^{-qs} \mathrm{d} R^{\pi}_s.
	\end{align*}
	In a similar way, by part (2) of Lemma \ref{P2}
	\begin{align*}
		\sum_{0 \leq s\leq t\wedge T_n} e^{-qs}[v_{\Lambda}({X}^\pi_{s-}+&\Delta [X_{s}+R^{\pi}_{s}]-\Delta D_s^\pi)-v_{\Lambda}({X}^\pi_{s-}+\Delta [X_{s}+R^{\pi}_{s}])]1_{\{\Delta D_s^\pi\not=0\}}\\
		&\leq-\sum_{0 \leq s\leq t\wedge T_n} e^{-qs} \Delta D_s^{\pi}+\delta\sum_{0 \leq s\leq t\wedge T_n} e^{-qs}1_{\{\Delta D_s^{\pi}>0\}}\\
		&=-\int_{[0, t\wedge T_n]}e^{-qs} \mathrm{d}\left( D^{\pi}_s-\delta\sum_{0 \leq u\leq s} 1_{\{\Delta D_u^{\pi}>0\}}\right).
	\end{align*}
	Hence, from \eqref{impulse_verif_1} we derive that
	\begin{align*}
		v_{\Lambda}(x) \geq&-\int_{0}^{t\wedge T_n}e^{-qs}  (\mathcal{L}-q)v_{\Lambda}({X}^\pi_{s-})  \mathrm{d}s-\Lambda\int_{[0, t\wedge T_n]}e^{-qs} \mathrm{d} R^{\pi}_s \\
		& + \int_{[0, t\wedge T_n]}e^{-qs} \diff\left( D^{\pi}_s-\delta\sum_{0 \leq u\leq s} 1_{\{\Delta D_u^{\pi}>0\}}\right) - M_{t\wedge T_n} + e^{-q(t\wedge T_n)}v_{\Lambda}({X}^\pi_{t\wedge T_n}).
	\end{align*}
	Using  Proposition \ref{generator_HJB} along with point 3 in the proof of Lemma 6 in \cite{LoeffenTrans}, and that $X_{s-}^\pi \geq 0$ a.s.\ for all $s \geq 0$, we have
	\begin{align} \label{w_lower}
		v_{\Lambda}(x) &\geq  \int_{[0, t\wedge T_n]}e^{-qs} \mathrm{d}\left( D^{\pi}_s-\delta\sum_{0 \leq u\leq s}1_{\{\Delta D_u^{\pi}>0\}}\right) -\Lambda\int_{[0, t\wedge T_n]}\mathrm{d} R^\pi_s- M_{t\wedge T_n}+  e^{-q(t\wedge T_n)}v_{\Lambda}({X}^\pi_{t\wedge T_n})\\\nonumber
		&\geq \int_{[0, t\wedge T_n]}e^{-qs} \mathrm{d}\left( D^{\pi}_s-\delta\sum_{0 \leq u\leq s}1_{\{\Delta D_u^{\pi}>0\}}\right) -\Lambda\int_{[0, t\wedge T_n]}\mathrm{d} R^\pi_s- M_{t\wedge T_n}\\\nonumber
		&\quad+  e^{-q(t\wedge T_n)}\left(\dfrac{\Lambda\psi'(0+)}{q}+Z^{(q)}(c_2^{\Lambda})\left(\zeta(c_2^{\Lambda})\right)\right),
	\end{align} 
	where the last inequality follows from Remark \ref{lower_bound}. 	In addition by the compensation formula (cf.\ Corollary 4.6 of \cite{kyprianou2014}), $(M_{t \wedge T_n}:t\geq0 )$ is a zero-mean $\mathbb{P}_x$-martingale. Now, taking expectations  in \eqref{w_lower} and
	letting $t$ and $n$ go to infinity ($T_n\nearrow\infty$ $\mathbb{P}_x$-a.s.), the monotone convergence theorem, applied separately for $\E_x [\int_{[0, t\wedge T_n]}e^{-qs} \mathrm{d}\left( D^{\pi}_s-\delta\sum_{0 \leq u\leq s}1_{\{\Delta D_u^{\pi}>0\}}\right)]$ and $\E_x (\Lambda \int_{[0, t\wedge T_n]}e^{-qs}\mathrm{d} R^\pi_s $), gives
	%\red{[with the current definition of $T_n$ it does not go to infinity]}
	\begin{equation*}
		v_{\Lambda}(x) \geq \mathbb{E}_x \left(\int_{[0, \infty]}e^{-qs} \mathrm{d}\left( D^{\pi}_s-\delta\sum_{0 \leq u\leq s} 1_{\{\Delta D_u^{\pi}>0\}}\right)-\Lambda \int_{[0,\infty)}e^{-qs}\mathrm{d} R^\pi_s \right) =v^{\pi}_{\Lambda}(x).
	\end{equation*}
	
	%Hence we proved $w(x)\geq v(x)$ for all $x\geq0$.
	This completes the proof.
\end{proof}

\section{Optimal dividends with capital injection constraint}\label{const.1}

In this section we are interested in maximizing the expected present value of the dividend strategy subject to a constraint in the expected present value of the injected capital. Specifically, we aim to solve 
\begin{equation}\label{p}
V(x,K):=\sup_{\pi\in\Theta}\e_x\left[\int_0^{\infty}\expo^{-qt}\ud \left(D_t^{\pi}-\delta\sum_{0\leq s<t}\mathbf{1}_{\{\Delta D_s^{\pi}>0\}}\right)\right]\quad\text{s.t. } \e_x\left[\int_0^{\infty}\expo^{-qt}\ud R_t\right]\leq K,
\end{equation}
for any $x\geq0$ and $K\geq0$. Strategies $\pi$ that do not satisfy the capital injection constraint are called infeasible. Recall that the beneficiaries of the dividends are imposed to inject capital in order to ensure the non-negativity of the risk process. Therefore, small values of $K$ will impose few capital injection, which will require almost none dividend payments to maintain the risk process non-negative, or even will make the problem infeasible. So this constraint adds a longevity feature to the problem. In the latter case we define the value function as $-\infty$.

In order to solve this problem we will make use of the solution of the optimal dividend problem with capital injection found above. So, for $\Lambda\geq0$ we define the function
\begin{equation}\label{lm.dual.1}
v_{\Lambda}^{\pi}(x,K):=v_{\Lambda}^{\pi}(x)+\Lambda K,
\end{equation}
with $v_{\Lambda}^{\pi}(x)$ as in \eqref{v.f.strategy}. It is easy to see that $V(x,K)=\sup\limits_{\pi\in\Theta}\inf\limits_{\Lambda\geq0}v_{\Lambda}^{\pi}(x,K)$ since for infeasible strategies $\inf\limits_{\Lambda\geq0}v_{\Lambda}^{\pi}(x,K)=-\infty$. By interchanging the $\sup$ with the $\inf$ we obtain an upper bound for $V(x,K)$, the so-called weak duality. Hence, the dual problem of \eqref{p} is defined as
\begin{equation}\label{p_dual}
V^D(x,K):=\inf_{\Lambda\geq0}\sup_{\pi\in\Theta}v_{\Lambda}^{\pi}(x,K)=\inf_{\Lambda\geq0}\left\{\Lambda K+\sup_{\pi\in\Theta}v_{\Lambda}^{\pi}(x)\right\}=\inf_{\Lambda\geq1}\left\{\Lambda K+V_{\Lambda}(x)\right\},
\end{equation}
with $V_{\Lambda}$ given in \eqref{v.f}, and where the last equality follows from the fact that $V_{\Lambda}(x)$ is infinity for any $\Lambda<1$. The main goal will be to prove that $V^D(x,K)\leq V(x,K)$.

\subsection{No transaction cost}

In this subsection we considered the problem \eqref{p} without transaction cost, i.e. $\delta=0$. Recall from Subsection \ref{dividends_injection_no_cost} that optimal strategies are barrier strategies, that is, $V_{\Lambda}=v_{\Lambda}^{a_{\Lambda}}$ with $a_{\Lambda}$ defined in \eqref{opt_an}. Given $x\geq0$ and a barrier strategy at level $a>0$, the expected present value of the injected capital is given by the function
\begin{equation}\label{funPsi}
\Psi_x(a):=\e_x\left[\int_0^{\infty}\expo^{-qt}\ud R^0_t\right].
\end{equation}
Clearly, if $x>a$, then $\Psi_x(a)=\Psi_a(a)$. We also define
\begin{equation}\label{K.1}
\underline{K}_{x}:=\lim\limits_{a\rightarrow\infty}\Psi_x(a). 
\end{equation}
Using \eqref{eqinjectionn} and the properties of scale functions (see Remark \ref{remark_smoothness_zero} (3)) we have that
\begin{equation*}\label{Kxn}
	\underline{K}_{x}=-k^{(q)}(x)+\dfrac{Z^{(q)}(x)}{\Phi(q)}.
\end{equation*}
Note that $\underline{K}_{x}$ is the expected present value of the injected capital for the pay-nothing strategy $\pi_{PN}:=\{0,R^0\}$. Therefore, letting $a\rightarrow\infty$ in \eqref{vf_1}, it can be verified
\begin{equation*}\label{inf.1}
	v_{\Lambda}^{\pi_{PN}}(x,K)=\Lambda(K-\underline{K}_{x}).
\end{equation*}
Hence, if $K\geq \underline{K}_{x}$, then for any $x\geq0$
\begin{equation}\label{inf.2}
V(x,K)=\sup_{\pi\in\Theta}\inf_{\Lambda\geq0}v_{\Lambda}^{\pi}(x,K)\geq\inf_{\Lambda\geq0}v_{\Lambda}^{\pi_{PN}}(x,K)=0.
\end{equation}
Conversely, if $K<\underline{K}_x$, the problem \eqref{p} is infeasible as it is shown next.
\begin{lemma}\label{lem_infeasible}
	If $K<\underline{K}_x$, then $V(x,K)=-\infty$.
\end{lemma}
\begin{proof}
	First of all, by Remark \ref{rem_optbar_inf} and \eqref{eqdividendn}, it is clear that for any $x\geq0$,
	\begin{equation}\label{lim1}
	\displaystyle\lim_{\Lambda\rightarrow\infty}\e_x\left[\int_0^{\infty}\expo^{-qt}\ud D^{a_{\Lambda}}_t\right]=0.
	\end{equation}
	Then,
	\begin{align*}
		V^D(x,K)&=\inf_{\Lambda\geq1}\left\{\Lambda K+V_{\Lambda}(x)\right\}\\
		&= \inf_{\Lambda\geq1}\bigg\{\e_x\left[\int_0^{\infty}\expo^{-qt}\ud D^{a_{\Lambda}}_t\right]+\Lambda(K-\Psi_x(a_{\Lambda}))\bigg\}\\
		&\leq\lim_{\Lambda\rightarrow\infty} \bigg\{\e_x\left[\int_0^{\infty}\expo^{-qt}\ud D^{a_{\Lambda}}_t\right]+\Lambda(K-\Psi_x(a_{\Lambda}))\bigg\}=-\infty.
	\end{align*}
Now, since $V(x,K)\leq V^D(x,K)$ for any $x\geq0$, $K\geq0$, we have the result.
\end{proof}
Next lemma allows us to prove that when $K=\underline{K}_x$ equality holds  in \eqref{inf.2}, and it will be used in the proof of the main result of this subsection.
\begin{lemma}\label{Psix}
	Let $x\geq0$ be fixed. The function $\Psi_x$ is strictly decreasing on $(0,\infty)$. 
\end{lemma}
\begin{proof}
	Consider first the case when $x<a$, then by Remark \ref{logconv} (i) we have that $\dfrac{qW^{(q)}(a)}{Z^{(q)}(a)}$ is strictly increasing and the result follows. Now, when $x\geq a>0$ a simple calculation shows that
	\begin{align*}
		\dfrac{\ud \Psi_a(a)}{\ud a}&=-\dfrac{Z^{(q)}(a)\left(W^{(q)\prime}(a)Z^{(q)}(a)-q[W^{(q)}(a)]^2\right)}{q[W^{(q)}(a)]^2}=-\dfrac{Z^{(q)}(a)W^{(q)\prime}(a)}{q[W^{(q)}(a)]^2}H(a),
	\end{align*}
	which is strictly negative, by Remarks \ref{remark_smoothness_zero} and \ref{remH}.%\blue{[Que tal solo dejarlo como ¨is strictly negative by Remarks \ref{remark_smoothness_zero} and \ref{remH}"?]}.
\end{proof}
\begin{lemma}\label{lem_do-not}
	If $K=\underline{K}_x$, then $V(x,K)=0$ and the optimal strategy is the pay-nothing strategy $\pi_{PN}$.
\end{lemma}
\begin{proof}
	By \eqref{inf.2}, we know that $V(x,K)\geq0$. On the other hand, from  Lemma \ref{Psix} and \eqref{K.1}, we have that $\Lambda(K-\Psi_{x}(a_{\Lambda}))\leq0$ for all $\Lambda\geq0$.	Then, using \eqref{K.1} and \eqref{lim1}
	\begin{align*}
			V^D(x,K)=\inf_{\Lambda\geq1}\left\{\Lambda K+V_{\Lambda}(x)\right\}&= \inf_{\Lambda\geq1}\bigg\{\e_x\left[\int_0^{\infty}\expo^{-qt}\ud D^{a_{\Lambda}}_t\right]+\Lambda(K-\Psi_x(a_{\Lambda}))\bigg\}\\
			&\leq\lim_{\Lambda\rightarrow\infty} \e_x\left[\int_0^{\infty}\expo^{-qt}\ud D^{a_{\Lambda}}_t\right]=0.
		\end{align*}
\end{proof}

Now, we define
\begin{equation}\label{K.2}
\overline{K}:=\lim\limits_{a\rightarrow0}\Psi_a(a). 
\end{equation}
Hence, $\overline{K}$ is the expected present value of the injected capital for the strategy $\pi_{0,0}$. Again, using equation \eqref{eqinjectionn} we have that $\overline{K}=\infty$ when the risk process has unbounded variation. Otherwise by Remark \eqref{remark_smoothness_zero} (2),
\begin{equation}\label{Kn}
\overline{K}=\dfrac{c-\psi'(0+)}{q}.
\end{equation}

\begin{lemma}\label{lem_bounded}
	Assume that the risk process $X$ has bounded variation. If $K\geq\overline{K}$, then $V(x,K)=K+V_1(x)$, with $V_1(x)=x+\dfrac{\psi'(0+)}{q}$.
\end{lemma}
\begin{proof}
	If the \lev measure is finite, by \eqref{opt_an} we have that $a_1=0$. The same is true when the \lev measure is infinite since $H^{-1}(1)=0$ by Remark \ref{remH}. Using \eqref{vf_1} and Remark \ref{remzeta} we obtain 
		\begin{equation}\label{stra.0}
			V_{1}(x)=v^{0}_{1}(x)=x+\dfrac{c}{q}-\overline{K}=x+\dfrac{\psi'(0+)}{q},\quad \text{for}\ x\geq0.
		\end{equation}
		Now, by \eqref{p_dual}, \eqref{Kn}, \eqref{stra.0} and weak duality we get 
			$$
			V(x,K)\leq V^{D}(x,K)\leq K+v^{0}_{1}(x)= K+V_1(x).$$
		To prove the equality we note that since $K\geq\overline{K}$, it follows that $\pi_{0,0}$ is a non-infeasible strategy. Then, it yields, using \eqref{eqdividendn},
		\begin{equation*}
			V(x,K)\geq \inf_{\Lambda\geq1}\{v_{\Lambda}^{0}(x)+\Lambda K\}=x+K-\dfrac{c-\psi'(0+)}{q}+\dfrac{c}{q}=K+V_1(x).
		\end{equation*}
		Therefore, $V(x,K)=K+V_{1}(x).$
\end{proof}

We are now ready for the main result of this subsection.
\begin{theorem}\label{strdualnocost}
	Assume $\delta=0$ and let $V$ and $V^D$ as in \eqref{p} and $\eqref{p_dual}$, respectively, then $V=V^D$. Furthermore, if $x,K$ are such that $K\in(\underline{K}_x,\overline{K})$, then
	\begin{equation}\label{main_nocost}
	V(x,K)=\Lambda^*K+V_{\Lambda^*}(x)=\e_x\left[\int_0^{\infty}\expo^{-qt}\ud D^{a^*}_t\right],
	\end{equation}
	where $a^*=\Psi_x^{-1}(K)$, and $\Lambda^*=\dfrac{1}{H(a^*)}$.
\end{theorem}
\begin{proof}
	Lemmas \ref{lem_infeasible}, \ref{lem_do-not} and \ref{lem_bounded} show the result for $x,K$ such that $K\in[0,\underline{K}_x]\cup[\overline{K},\infty)$. Assume now that $K\in(\underline{K}_x,\overline{K})$, then by Lemma \ref{Psix} the function $\Psi_x$ is injective, so there exists a unique $a^*>0$ such that $\Psi_x(a^*)=K$. Note that from the expression \eqref{opt_an}, we have that there exists a unique $\Lambda^{*}$ such that $a_{\Lambda^*}=a^*$. %Hence the strategy $\pi_{a^*,0}$ is a feasible strategy, since 
	%\begin{equation*}
	%\inf_{\Lambda\geq1}v^{\pi_{a^{*},0}}_{\Lambda}(x,K)\leq v_{\Lambda^{*}}^{a_{\Lambda^{*}}}(x)+\Lambda^{*}K=\e_x\left[\int_0^{\infty}\expo^{-qt}\ud D^{a_{\Lambda^{*}}}_t\right]<\infty \blue{>-\infty?}.
	%\end{equation*}
	Then,
	\begin{align*}
		V^D(x,K)&\leq\Lambda^*K+V_{\Lambda^*}(x)\\
		&=\Lambda^*K+\e_x\left[\int_0^{\infty}\expo^{-qt}\ud D^{a^*}_t\right]-\Lambda^*\Psi_x(a^*)\\
		&=\e_x\left[\int_0^{\infty}\expo^{-qt}\ud D^{a^*}_t\right].
	\end{align*}
	On the other hand, since the strategy $\pi_{a^*,0}$ is feasible, we see  
	\begin{align*}
	V(x,K)\geq \inf_{\Lambda\geq1}\big\{v^{\pi_{a^{*},0}}_{\Lambda}(x)+\Lambda K\big\}&=\inf_{\Lambda\geq1}\Big\{\e_x\left[\int_0^{\infty}\expo^{-qt}\ud D^{a^*}_t\right]+\Lambda (K-\Psi_{x}(a^{*}))\Big\}\\&=\e_x\left[\int_0^{\infty}\expo^{-qt}\ud D^{a^*}_t\right].
	\end{align*}
This implies that $V^{D}(x,K)\leq V(x,K)$. Finally, weak duality gives \eqref{main_nocost}. 
\end{proof}

\subsection{With transaction cost}\label{dividends_injection_cost}

Now we consider the problem given in \eqref{p} with transaction cost $\delta>0$. From the previous section we know that optimal strategies are $(c_1^\Lambda,c_2^\Lambda)$-reflected strategies with $(c_1^\Lambda,c_2^\Lambda)$ given in Proposition \ref{remoptc}.
\begin{proposition}\label{lT}
	The curve $\Lambda\mapsto(c^{\Lambda}_{1},c^{\Lambda}_{2})$ for $\Lambda\geq1$ is continuous and unbounded.
\end{proposition}
\begin{proof}
	From Remark \ref{rem_optbar_inf} and the fact that $a_{\Lambda}<c^{\Lambda}_{2}$ (by Proposition \ref{remoptc}), we know that $c^{\Lambda}_{2}\rightarrow\infty$ as $\Lambda\rightarrow\infty$, so the curve is unbounded. The continuity follows from the Implicit Function Theorem by considering two cases. First, suppose $c^{\Lambda}_{1}=0$. Defining $f(\Lambda,c_{2}):=G_{\Lambda}(0,c_{2})-\zeta_{\Lambda}(c_{2})$, we have $f(\Lambda,c^{\Lambda}_{2})=0$. Simple calculations show that 
	$$\dfrac{\partial f}{\partial  c_{2}}(\Lambda,c^{\Lambda}_{2})=\dfrac{\partial G_{\Lambda}}{\partial  c_{2}}(0,c^{\Lambda}_{2})-\zeta'_{\Lambda}(c^{\Lambda}_{2})=-\zeta'_{\Lambda}(c^{\Lambda}_{2})>0,$$
	since $c^{\Lambda}_{2}>a_{\Lambda}$. So the conditions of the Implicit Function Theorem are satisfied. Now, if $c^{\Lambda}_{1}>0$, define the function $f(\Lambda,c_{1},c_{2})=(f_1(\Lambda,c_{1},c_{2}),f_2(\Lambda,c_{1},c_{2}))$  by
	\begin{align*}
		f_1(\Lambda,c_{1},c_{2})&:=G_{\Lambda}(c_{1},c_{2})-\zeta_{\Lambda}(c_{1}),\\
		f_2(\Lambda,c_{1},c_{2})&:=G_{\Lambda}(c_{1},c_{2})-\zeta_{\Lambda}(c_{2}).
	\end{align*}
	Then $f(\Lambda,c^{\Lambda}_{1},c^{\Lambda}_{2})=(0,0)$. Again, simple calculations show that the Jacobian determinant of this system of equations is $\zeta'_{\Lambda}(c^{\Lambda}_{2})\zeta'_{\Lambda}(c^{\Lambda}_{1})<0$, since $c^{\Lambda}_{1}< a_{\Lambda}<c^{\Lambda}_{2}$, implying the continuity of the curve.
\end{proof}
Next, we proceed to analyze the level curves of the constraint. Let $\overline{\Psi}_{x}(c_{1},c_{2})$ be the expected present value of the injected capital under a $(c_1,c_2)$-reflected policy. Then, the calculations previous to Lemma \ref{P1} show that
\begin{align}\label{PsiTrans}\nonumber
	\overline{\Psi}_{x}(c_{1},c_{2}):&=\mathbb{E}_x\left[\int_0^{\infty}\expo^{-qt}\diff R_{t}\right]\\
	&=
	\begin{cases} 
		Z^{(q)}(x)\dfrac{\overline{Z}^{(q)}(c_{2})-\overline{Z}^{(q)}(c_{1})}{Z^{(q)}(c_{2})-Z^{(q)}(c_{1})}- k^{(q)}(x), & \mbox{if } 0\leq x\leq c_{2},\\
		\dfrac{\overline{Z}^{(q)}(c_{2})Z^{(q)}(c_{1})-\overline{Z}^{(q)}(c_{1})Z^{(q)}(c_{2})}{Z^{(q)}(c_{2})-Z^{(q)}(c_{1})}-\dfrac {\psi'(0+)}{q},  &   \mbox{if } x>c_{2}.
	\end{cases}
\end{align}

\begin{remark}\label{rel_psi_barpsi}
	Note that $\displaystyle\lim_{c_{1}\rightarrow c_{2}}\overline{\Psi}_{x}(c_{1},c_{2})=\Psi_{x}(c_{2})$, where $\Psi_x$ is defined in \eqref{funPsi}.
\end{remark}

The next few lemmas will describe some properties of $\overline{\Psi}_{x}(c_{1},c_{2})$.
\begin{lemma}\label{Psidecr}
	Let $x\geq0$ be fixed. 
	\begin{enumerate}
		\item  If $c_{1}\geq0$ is fixed, then the function $\overline{\Psi}_{x}(c_{1},c_{2})$, given in \eqref{PsiTrans}, is strictly decreasing  for all $c_{2}>c_{1}$, and
		\begin{equation}\label{KlimitTrans}
		\lim_{c_{2}\rightarrow\infty}\overline{\Psi}_{x}(c_{1},c_{2})=\underline{K}_{x},
		\end{equation}
		where $\underline{K}_x$ is defined in \eqref{K.1}.
		\item  If $c_{2}>0$ is fixed, $\overline{\Psi}_{x}(c_{1},c_{2})$ is strictly decreasing  for all $c_{1}\in[0,c_{2})$.
	\end{enumerate}
\end{lemma}

\begin{proof}
	Let $c_{1}\geq0$ be fixed. First, assume that $c_{2}\geq x$. To show that $\overline{\Psi}_{x}(c_{1},c_{2})$ is strictly decreasing, it is sufficient to verify that 
	\begin{equation}\label{p5.0}
	\dfrac{\overline{Z}^{(q)}(c_{2})-\overline{Z}^{(q)}(c_{1})}{Z^{(q)}(c_{2})-Z^{(q)}(c_{1})}
	\end{equation}
	is strictly decreasing, which is true if 
	\begin{align}\label{p5}
		\dfrac{\partial}{\partial c_{2}}\biggr[\dfrac{\overline{Z}^{(q)}(c_{2})-\overline{Z}^{(q)}(c_{1})}{Z^{(q)}(c_{2})-Z^{(q)}(c_{1})}\biggl]=\dfrac{Z^{(q)}(c_{2})}{Z^{(q)}(c_{2})-Z^{(q)}(c_{1})}-\dfrac{qW^{(q)}(c_{2})(\overline{Z}^{(q)}(c_{2})-\overline{Z}^{(q)}(c_{1}))}{[Z^{(q)}(c_{2})-Z^{(q)}(c_{1})]^{2}}<0.
	\end{align}
	Since $Z^{(q)}$ is a strictly log-convex function on $[0,\infty)$ by Remark \ref{logconv} (i), we have that 
	\begin{equation*}%\label{p4}
		\dfrac{qW^{(q)}(\eta)}{Z^{(q)}(\eta)}<\dfrac{qW^{(q)}(\varsigma)}{Z^{(q)}(\varsigma)},\ \text{for any}\ \eta \ \text{and}\ \varsigma\ \text{with}\ \eta<\varsigma. 
	\end{equation*}
	Taking $\varsigma=c_{2}$ in the above inequality and   integrating between $c_{1}$ and $c_{2}$, it follows that 
	\begin{equation}\label{p5.1}
	Z^{(q)}(c_{2})<\dfrac{qW^{(q)}(c_{2})[\overline{Z}^{(q)}(c_{2})-\overline{Z}^{(q)}(c_{1})]}{Z^{(q)}(c_{2})-Z^{(q)}(c_{1})}. 
	\end{equation}
	Then, it yields \eqref{p5} and hence \eqref{p5.0} is  strictly decreasing. For the case $x>c_{2}$, it can verified that  
	\begin{equation}\label{p5.2}
	\dfrac{\partial}{\partial c_{2}}\biggr[Z^{(q)}(c_{2})\dfrac{\overline{Z}^{(q)}(c_{2})-\overline{Z}^{(q)}(c_{1})}{Z^{(q)}(c_{2})-Z^{(q)}(c_{1})}-\overline{Z}^{(q)}(c_{2})\biggl]=\frac{Z^{(q)}(c_{1})}{Z^{(q)}(c_{2})-Z^{(q)}(c_{1})}\biggr[Z^{(q)}(c_{2})-\dfrac{qW^{(q)}(c_{2})[\overline{Z}^{(q)}(c_{2})-\overline{Z}^{(q)}(c_{1})]}{Z^{(q)}(c_{2})-Z^{(q)}(c_{1})}\biggl].
	\end{equation}
	Then, by \eqref{p5.1}--\eqref{p5.2}, it yields that $\overline{\Psi}_{x}(c_{1},c_{2})$ is strictly decreasing for all $c_{2}\in(c_{1},x)$. Proceeding in a similar way that before, we also obtain (2). Now, by L'H\^opital's rule together with Exercise 8.5 (i) in \cite{kyprianou2014}, it is not difficult to see that \eqref{KlimitTrans} holds for any $c_{1}\geq 0$. 
\end{proof}

Note that for $K\geq \underline{K}_{x}$ \eqref{inf.2} still holds in this case. On the other hand using that $c_2^{\Lambda}\to\infty$ as $\Lambda\to\infty$ together with \eqref{div_trans_c_2>x} we have that
\begin{align}\label{limcost}
\lim_{\Lambda\rightarrow\infty}\e_x\left[\int_0^{\infty}\expo^{-qt}\ud D^{c_1^{\Lambda},c_2^{\Lambda}}_t\right]=\lim_{\Lambda\rightarrow\infty}(c_2^{\Lambda}-c_1^{\Lambda}-\delta)\dfrac{Z^{(q)}(x)}{Z^{(q)}(c_2^{\Lambda})-Z^{(q)}(c_1^{\Lambda})}=0,
\end{align}
by Remark \ref{remark_smoothness_zero} (3).

\begin{remark}\label{boundedcost}
Using the same arguments as in Lemma \ref{lem_bounded} we have that  $c_1^1=a_1=0<c_2^1$ for both, bounded and unbounded variation processes. Similarly, if $x,K$ are such that $K\geq\overline{\Psi}_x(0,c_2^1)=:\overline{K}_x$, then $V(x,K)=V_1(x)+K$. Note also that $\overline{K}_x<\overline{K}$.
\end{remark}

%Note also that Lemmas \ref{lem_infeasible} and \ref{lem_do-not} are also valid by the previous lemma and \blue{the fact that}
%$$\lim_{\Lambda\rightarrow\infty}\e_x\left[\int_0^{\infty}\expo^{-qt}\ud D^{c_1^{\Lambda},c_2^{\Lambda}}_t\right]=0.$$

\begin{lemma}\label{Kcurve}
	Let $x\geq0$. Then, for each $K\in(\underline{K}_{x},\overline{K})$ there exist $\underline{c}\leq\overline{c}$ such that the level curve $L_K(\overline{\Psi}_x)=\{(c_{1},c_{2}):\overline{\Psi}_x(c_{1},c_{2})=K\}$ is continuous, contained in the set $[0,\underline{c}]\times[\underline{c},\overline{c}]$ and contains the points $(0,\overline{c})$ and $(\underline{c},\underline{c})$.
\end{lemma}
\begin{proof}
	The continuity of the level curve is an immediate consequence of the continuity of $\overline{\Psi}_{x}$. First, observe that by Lemma \ref{Psix} we know the existence of $\underline{c}>0$ such that $\Psi_{x}(\underline{c})=K$.  On the other hand, by Lemma \ref{Psidecr} there exists $\overline{c}\in[\underline{c},\infty)$ such that $\overline{\Psi}_{x}(0,\overline{c})=K$. Now, the fact that the level curve $L_{K}(\overline{\Psi}_{x})$ is contained in $[0,\underline{c}]\times[\underline{c},\overline{c}]$ is again consequence of Remark \ref{rel_psi_barpsi} together with Lemma \ref{Psidecr}.
\end{proof}

\begin{remark}\label{lopt}
By Lemmas \ref{lT} and \ref{Kcurve} we deduce that the parametric curve $\Lambda\mapsto (c_1^{\Lambda},c_2^{\Lambda})$ and the level curve $L_K(\overline{\Psi}_x)$ must intersect, that is, there exist $\Lambda^*$ such that $\overline{\Psi}_x(c_1^{\Lambda^*},c_2^{\Lambda^*})=K$, for $K\in(\underline{K}_{x},\overline{K}_x]$.
\end{remark}

By similar arguments that in the proof of Theorem \ref{strdualnocost} and using \eqref{limcost} and Remarks \ref{boundedcost} and \ref{lopt}, we get  the following result, whose proof is omitted.  
\begin{theorem}\label{main.1}
Assume $\delta>0$ and let $V$ and $V^D$ as in \eqref{p} and $\eqref{p_dual}$, respectively, then $V=V^D$. Furthermore, if $x,K$ are such that
\begin{enumerate}
\item $K<\underline{K}_x$, then $V(x,K)=-\infty$,
\item $K=\underline{K}_x$, then $V(x,K)=0$,
\item $K\geq\overline{K}_x$, then $V(x,K)=V_1(x)+K$,
\item $K\in(\underline{K}_{x},\overline{K}_x)$, then there exists $\Lambda^*\geq 1$ such that
\begin{equation}\label{main_nocost}
	V(x,K)=\Lambda^*K+V_{\Lambda^*}(x)=\mathbb{E}_x\left[\int_0^{\infty}\expo^{-qt}\diff\left(D^{c^{\Lambda^*}_1,c^{\Lambda^*}_2}_t-\delta\sum_{0\leq s<t}\mathbf{1}_{\{\Delta D^{c^{\Lambda^*}_1,c^{\Lambda^*}_2}_s>0\}}\right)\right].
	\end{equation}
\end{enumerate}
\end{theorem}

\section{Numerical examples}\label{numerical_section}

In this section, we confirm the obtained results through a sequence of numerical examples. Here we assume that $X$ is of the form
\begin{equation}
 X_t - X_0= t+ 0.5 B_t - \sum_{n=1}^{N_t} Z_n, \quad 0\le t <\infty, \label{X_gamma}
\end{equation}
where $B=\{B_t : t\ge 0\}$ is a standard Brownian motion, $N=\{N_t: t\ge 0\}$ is a Poisson process with arrival rate $\lambda=0.4$, and  $Z=\{Z_n; n = 1,2,\ldots\}$ is an i.i.d.\ sequence of random variables with distribution Gamma(1,2). Here, the processes $B$, $N$, and $Z$ are assumed mutually independent.  Since there is no close form for the scale functions associated with $X$, we will follow a numerical procedure presented in \cite{surya2008} to approximate the scale functions by Laplace transform inversion of \eqref{scale_function_laplace}. We do the same to approximate derivatives of the scale functions and use the trapezoidal rule to calculate integrals of it.

We first consider the case without transaction cost as studied in Subsection \ref{dividends_injection_no_cost}. In the left panel of Figure  \ref{figure_Lagrange_2d}, we plot the function $x \mapsto V_\Lambda(x) + \Lambda K$ for various values of $\Lambda$ and a fixed value of $K$. For $x\geq x_0$, where $x_0$ is such that $\underline{K}_{x_0}=K$, its minimum over the considered $\Lambda$ gives (an approximation of) $V(x, K)$, indicated by the solid red line in the plot. Since the process has unbounded variation, then $\overline{K}=\infty$. On the right panel of Figure \ref{figure_Lagrange_2d}, we plot, for $x>x_0$, the Lagrange multiplier $\Lambda^* $ given in Theorem \ref{strdualnocost}. We observe that $\Lambda^*$ goes to infinity as $x \downarrow x_0$ and remains always above 1.

\begin{figure}[t!]
   \begin{subfigure}[b]{.43\linewidth}
      \includegraphics[width=1\textwidth]{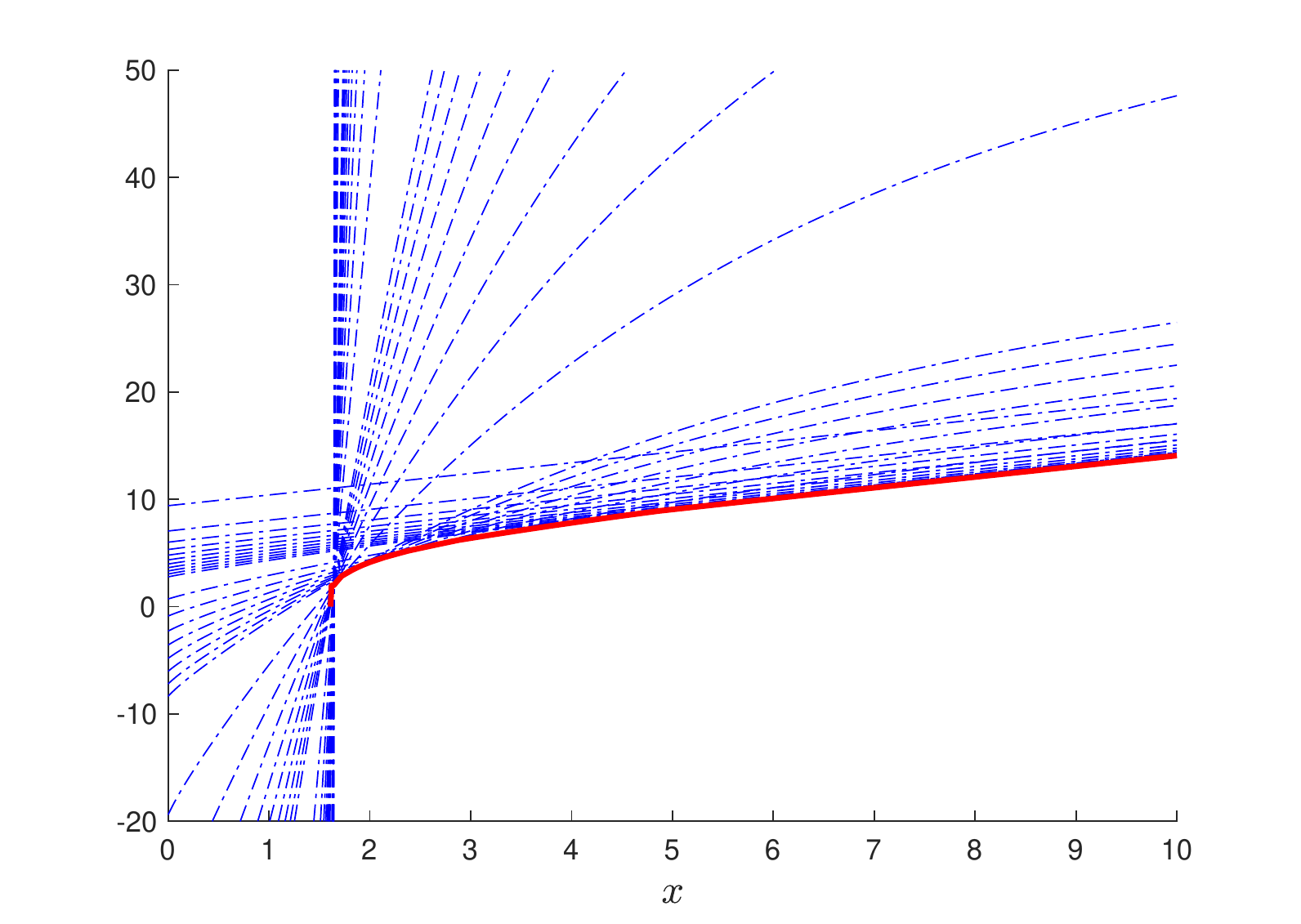}
   \end{subfigure}
\begin{subfigure}[b]{.43\linewidth}
      \includegraphics[width=1\textwidth]{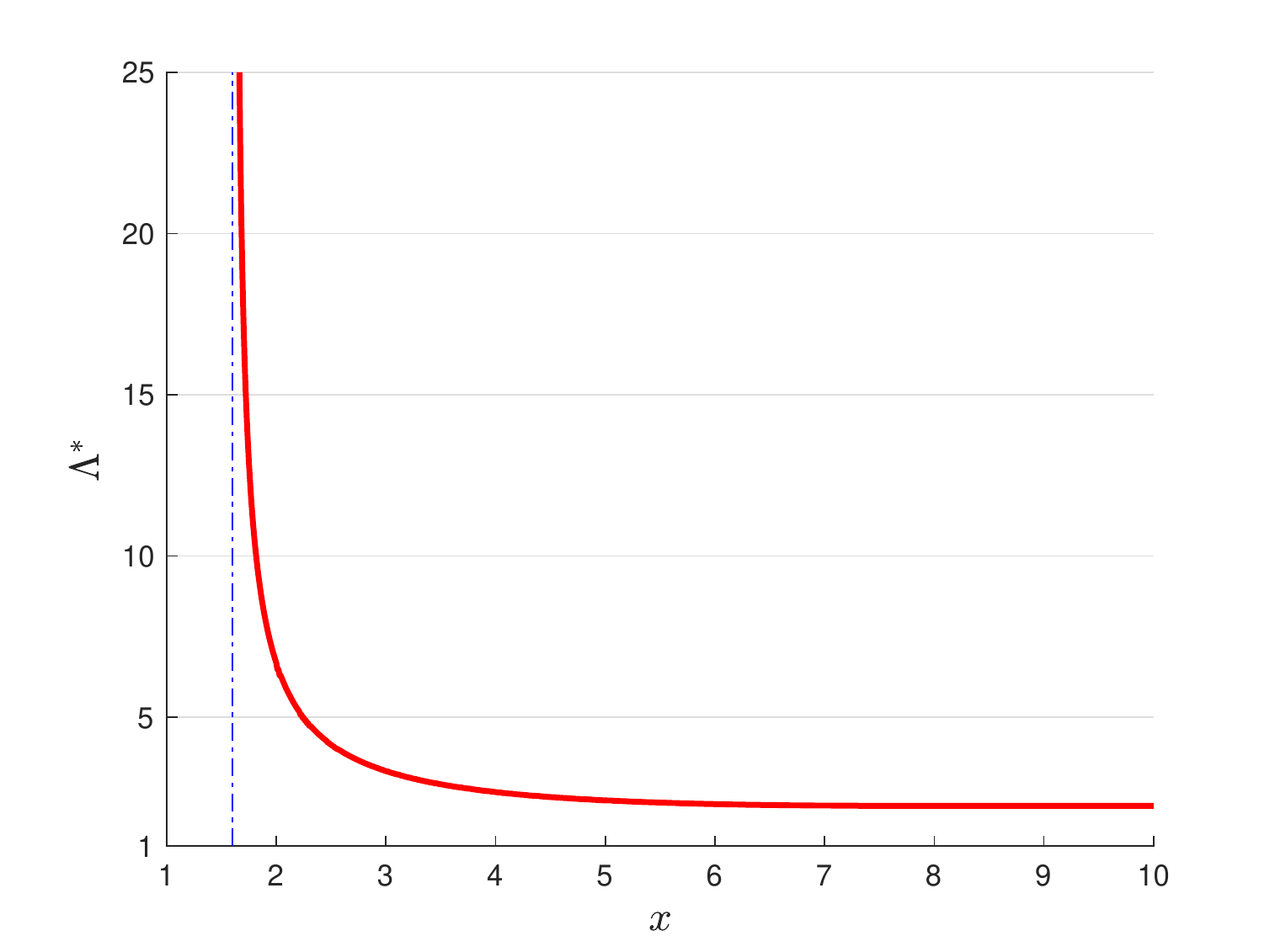}
   \end{subfigure}
      \caption{(Left) Plots of $x \mapsto V_\Lambda(x)+\Lambda K$ for $\Lambda = 1,1.1,\ldots,2,3,\ldots,10,$ $ 20, \ldots, 100, 200,\ldots,1000, 2000, \ldots,10000, 20000$ (dotted) for the case $K = 2.7$.  The minimum of $V_\Lambda(x)+\Lambda K$ over $\Lambda$ is shown in solid fold-face red line. (Right) Plot of the Lagrange multiplier $\Lambda^*$ for $x>x_0$, where $x_0$ is such that $\underline{K}_{x_0}=K$.}\label{figure_Lagrange_2d}
\end{figure}

In Figure \ref{figure_Lagrange_3d}, we show the values of $V(x, K)$ and Lagrange multiplier $\Lambda^*$ as functions of $(x, K)$. It is confirmed that $V(x; K)$ increases as $x$ and $K$ increase, while  $\Lambda^*$ increases as $(x, K)$ decreases.

\begin{figure}[t!]
   \begin{subfigure}[b]{.43\linewidth}
      \includegraphics[width=1\textwidth]{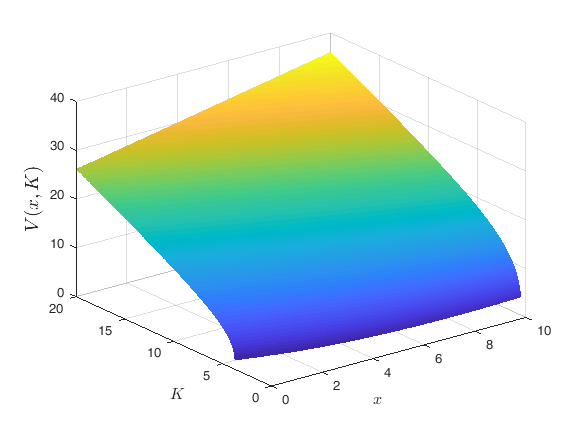}
   \end{subfigure}
\begin{subfigure}[b]{.43\linewidth}
      \includegraphics[width=1\textwidth]{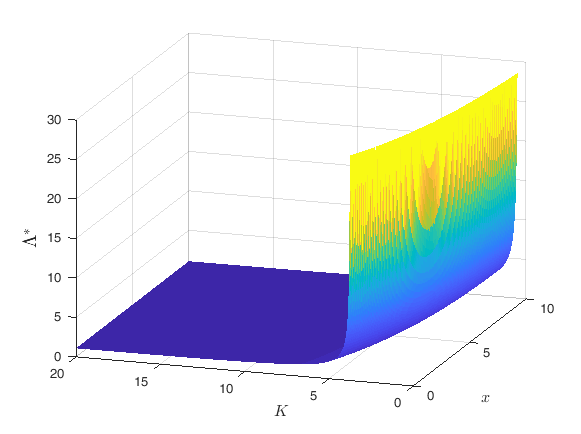}
   \end{subfigure}
      \caption{
Plots of $V(x; K)$ (left) and the Lagrange multiplier $\Lambda^*$ (right) as functions of $x$ and $K$.}\label{figure_Lagrange_3d}
\end{figure}

We now move to the case with transaction cost. First, we illustrate the results showed in Section \ref{neg.1}. In the left panel of Figure \ref{copt}, we plot the function $x\mapsto \zeta_\Lambda(x)$ for the values of $\Lambda=1\ldots,9$. We also plot its maximum value attained at $a_\Lambda$ and the value attained at the corresponding optimal values $(c_1^\Lambda,c_2^{\Lambda})$ with transaction cost $\delta=0.05$. Note that when $\Lambda=1$, $a_\Lambda=c_1^\Lambda=0$ and for the other values of $\Lambda$, $\zeta_{\Lambda}(c_1^\Lambda)=\zeta_{\Lambda}(c_2^\Lambda)<\zeta_{\Lambda}(a_{\Lambda})$. In the right panel of the figure we plot the optimal thresholds $a_{\Lambda},c_1^\Lambda$ and $,c_2^\Lambda$ as function of $\Lambda$.

\begin{figure}[t!]
   \begin{subfigure}[b]{.43\linewidth}
      \includegraphics[width=1\textwidth]{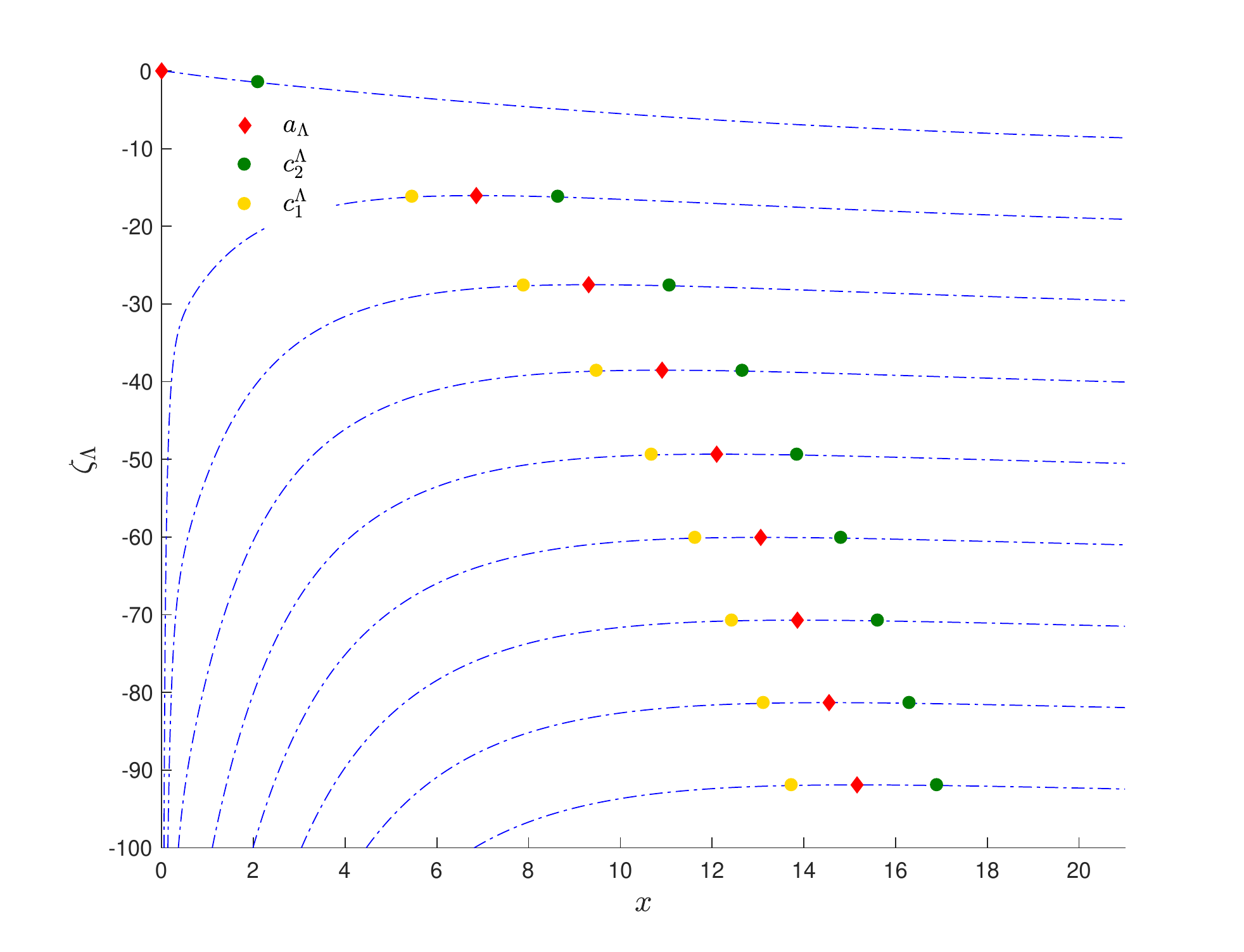}
   \end{subfigure}
\begin{subfigure}[b]{.43\linewidth}
      \includegraphics[width=1\textwidth]{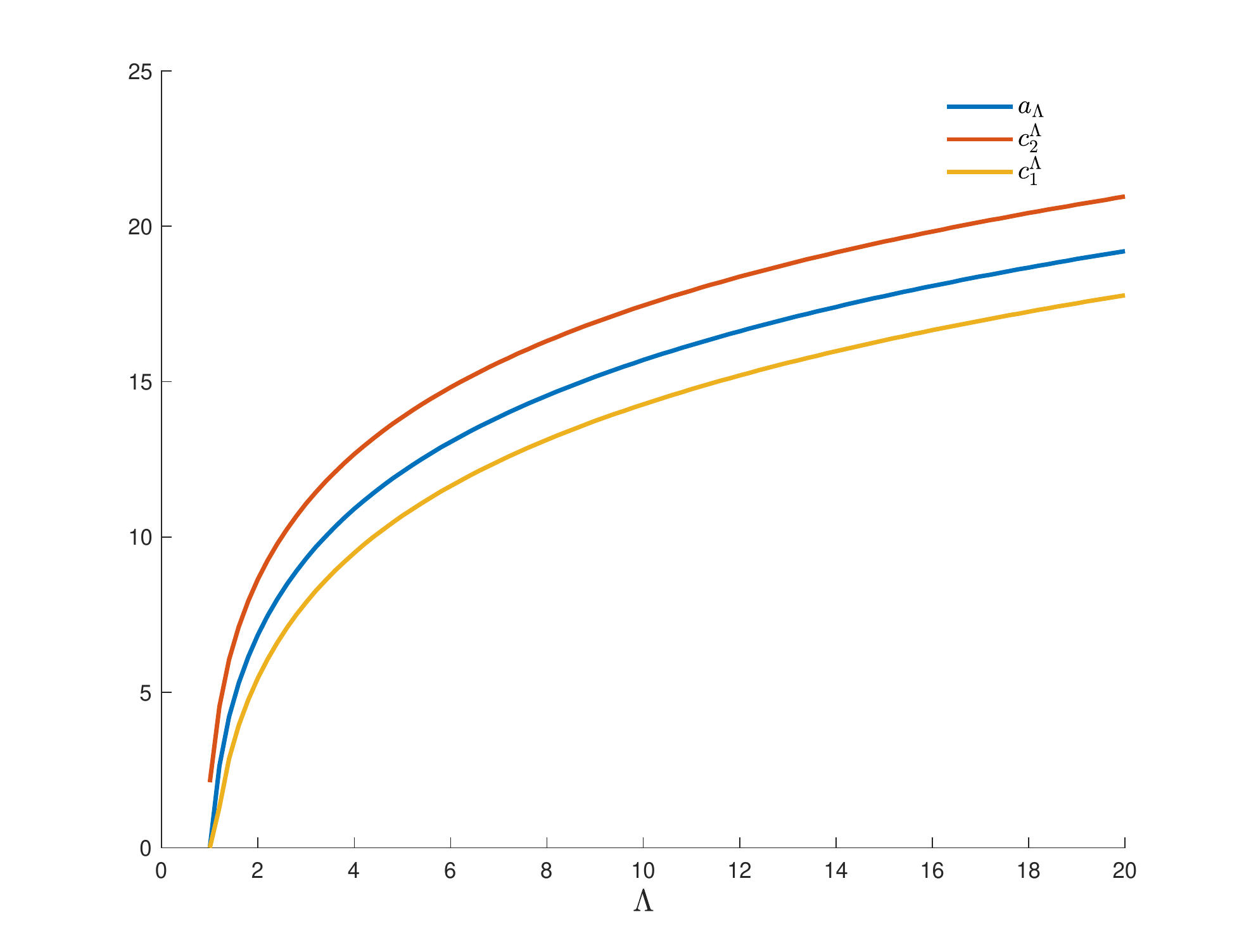}
   \end{subfigure}
      \caption{(Left) Plots of $x\mapsto \zeta_\Lambda(x)$ for $\Lambda=1,\ldots,9$ and the corresponding values of $a_{\Lambda},c_1^\Lambda$ and $,c_2^\Lambda$ for $\delta=0.05$. (Right) Plots of the functions $\Lambda\mapsto a_{\Lambda},c_1^\Lambda$ and $,c_2^\Lambda$.}\label{copt}
\end{figure}

In Figure \ref{figure_Lagrange_2dcost} we illustrate the findings of Subsection \ref{dividends_injection_cost}. This figure is analogous to Figure \ref{figure_Lagrange_2d} but with transaction cost $\delta$ as above. It can be seen that the change in the function $V(x,K)$ is relatively very small, but the change in the optimal Lagrange multiplier $\Lambda^*$ is significant, being smaller in the case of transaction cost. A similar figure as Figure \ref{figure_Lagrange_3d} in the case of transaction cost is omitted since both have the same shape.

\begin{figure}[t!]
   \begin{subfigure}[b]{.43\linewidth}
      \includegraphics[width=1\textwidth]{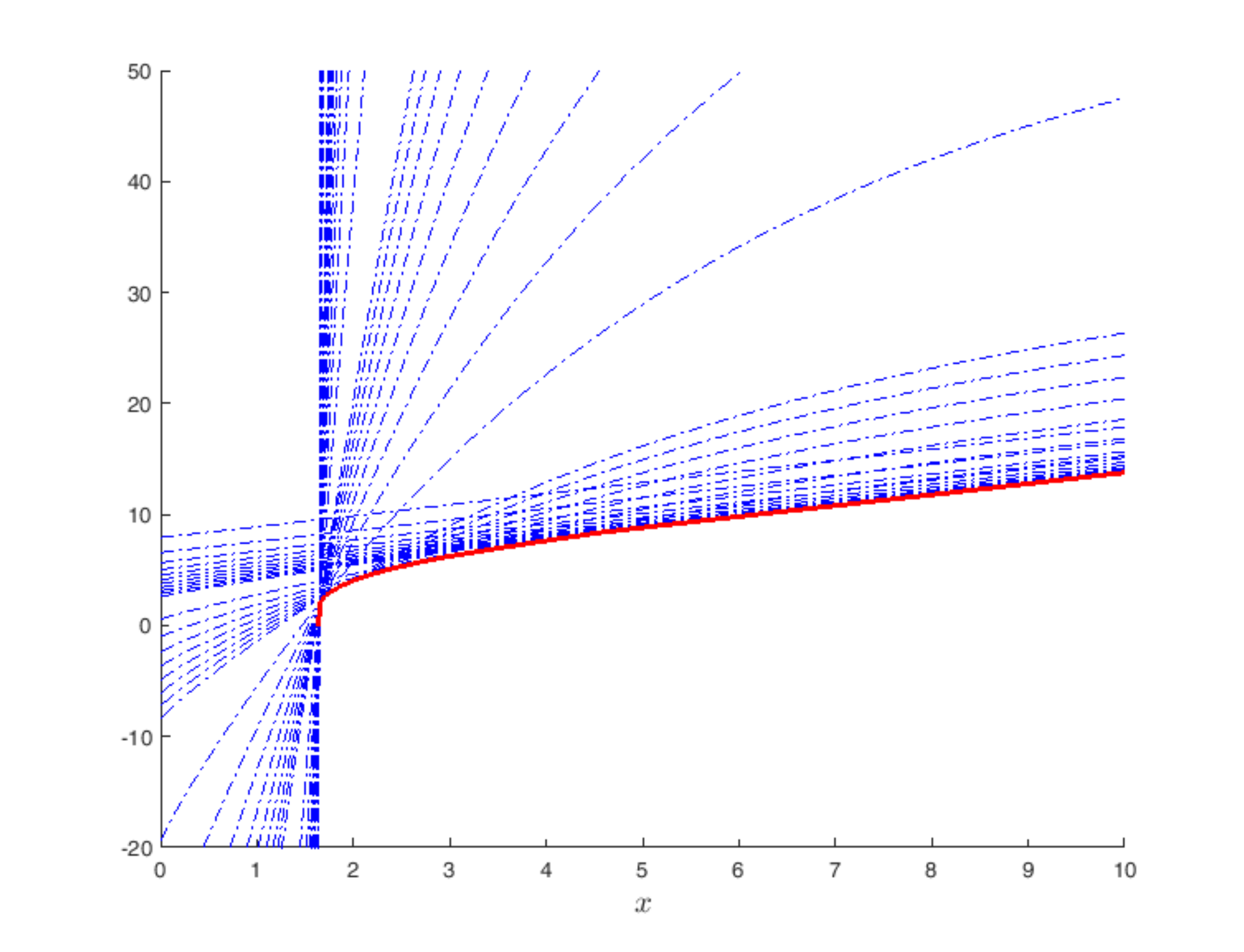}
   \end{subfigure}
\begin{subfigure}[b]{.43\linewidth}
      \includegraphics[width=1\textwidth]{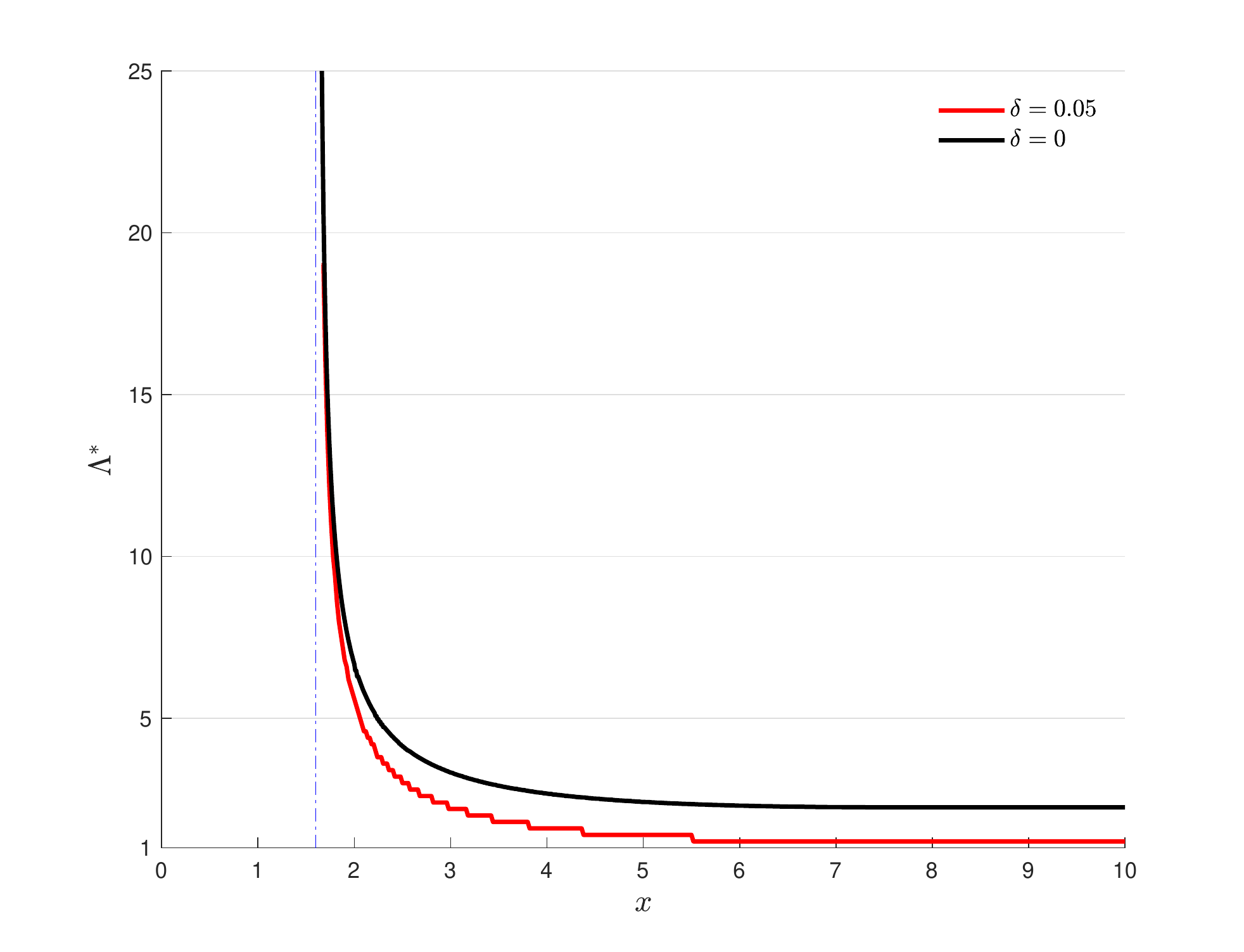}
   \end{subfigure}
      \caption{(Left) Plots of $x \mapsto V_\Lambda(x)+\Lambda K$ for $\Lambda = 1,1.1,\ldots,2,3,\ldots,10,$ $ 20, \ldots, 100, 200,\ldots,1000, 2000, \ldots,10000, 20000$ (dotted) for the case $K = 2.7$.  The minimum of $V_\Lambda(x)+\Lambda K$ over $\Lambda$ is shown in solid fold-face red line. (Right) Plots of the Lagrange multipliers $\Lambda^*$ for $x>x_0$, where $x_0$ is such that $\underline{K}_{x_0}=K$ with $\delta=0$ and $\delta=0.05$. }\label{figure_Lagrange_2dcost}
\end{figure}
\bibliographystyle{acm} 
\bibliography{ref}

\end{document}